\documentclass[12pt,a4paper]{amsart}

\usepackage{latexsym} 
 
\usepackage{graphicx}
\usepackage{epsfig}
\usepackage{mathrsfs}
\usepackage{amssymb}
\usepackage{amsthm}
\usepackage{amsfonts}
\usepackage{amsmath}
\usepackage{centernot}
\usepackage{amstext}
\usepackage{amscd}
\usepackage{enumerate}
\usepackage{tikz}
\usepackage{epic}
\usepackage{eepic}
\usepackage{pstricks,pst-plot}
\usepackage{bigdelim}
\usepackage{multirow}

\numberwithin{equation}{section}

\theoremstyle{plain}
\newtheorem{thm}{Theorem}[section] 
\newtheorem{prop}[thm]{Proposition}
\newtheorem{cor}[thm]{Corollary}
\newtheorem{lem}[thm]{Lemma}

\newtheorem{theorem*}{Theorem}[]

\theoremstyle{definition}
\newtheorem{defn}[thm]{Definition}

\newtheorem{example}[thm]{Example}

\theoremstyle{remark}
\newtheorem{rmk}[thm]{Remark}
\theoremstyle{property}

\newcommand{\N}{\mathbb{N}}
\newcommand{\R}{\mathbb{R}}

\DeclareMathOperator{\grad}{grad\,}
\DeclareMathOperator{\dist}{dist\,}

\def\accentclass@{7}
\def\makeacc@#1#2{\def#1{\mathaccent"\accentclass@#2 }}
\makeacc@\cir{017}

\newcommand{\s}{\Sigma}
\newcommand{\vp}{\varphi}

\def\dis{\displaystyle}

\def\det{\mathop{\rm det}}

\def\min{\mathop{\rm min}}

\newcommand{\ds}{d(x,\s)}
\newcommand{\de}{\delta}
\newcommand{\sing}{\mathrm{Sing}}

\def\det{{\text {\rm det}}}

\def\a {{\alpha}}

\def\d {{\delta}}

\def\vp{{\varphi}}

\title[Characterisations of sufficiency of relative jets]
{Characterisations of $V$-sufficiency and \\
$C^0$-sufficiency of relative jets}

\author{Karim Bekka and Satoshi Koike} 

\address{Institut de recherche Math\'ematique de Rennes, 
Universit\'{e} de Rennes 1, Campus Beaulieu, 35042 Rennes cedex, France}
\address{Department of Mathematics, Hyogo University of Teacher Education,
Kato, Hyogo 673-1494, Japan}
\email{karim.bekka@univ-rennes1.fr}
\email{koike@hyogo-u.ac.jp} 

\date{\today}

\begin{document}

\thanks{This research is partially supported by the Grant-in-Aid 
for Scientific Research (No. 23540087, 26287011) of Ministry of Education, 
Science and Culture of Japan, and HUTE Short-Term Fellowship
Program 2012 \& 2016.}


\maketitle

\begin{abstract}
We consider the problems of sufficiency of jets relative to a given closed set.
In the non-relative case, criteria for $r$-jets to be $V$-sufficient and 
$C^0$-sufficient in $C^r$ mappings or $C^{r+1}$ mappings have been obtained. 
In particular, it is shown that $V$-sufficiency and $C^0$-sufficiency in 
$C^r$ functions or $C^{r+1}$ functions are equivalent.

In this paper we discuss characterisations of $V$-sufficiency and 
$C^0$-sufficiency in the relative case, corresponding to the above 
non-relative results.
Applying the results obtained in the relative case, we construct examples of 
polynomial functions whose relative $r$-jets are $V$-sufficient in $C^r$
functions and $C^{r+1}$ functions but not $C^0$-sufficient in $C^r$ functions
and $C^{r+1}$ functions, respectively.

In addition, we give characterisations of relative finite $V$-determinacy
and also relative finite $C^r$ contact determinacy.
\end{abstract}

\bigskip

\section{Introduction}

Sufficiency of jets is one of the most important notions introduced by Ren\'e 
Thom for the structural stability theory.
The property of sufficiency of $r$-jet is a kind of
local stability at degree higher than $r$.
Implicit Function Theorem and Morse Lemma may be regarded as results
on sufficiency.
The notion of sufficiency of jets also has applications to the bifurcation 
problems in Differential Equation.
Hence this notion has been explored by many researchers in the 1970s and 
1980s (see C. T. C. Wall \cite{wall} for the survey of this field).

Any mapping realisation of a $C^0$-sufficient jet has an isolated 
singularity, and the zero-set of any mapping realisation of a 
$V$-sufficient jet also has an isolated singularity.
Therefore the above works on sufficiency of jets only deal with 
the isolated singularity case.
On the other hand, the works on characterisations of sufficiency of jets 
relative to a given closed set have been also started, e.g.  V. Grandjean
\cite{grandjean}, S. Izumiya and S. Matsuoka \cite{Izumiya-Matsuoka},
L. Kushner and B. Terra Leme \cite{Kushner-Terra Leme}, V. Thilliez
\cite{Thilliez}, X. Xu \cite{Xu},  P. Migus, T. Rodak and S. Spodzieja \cite{Migus-Rodak-Spodzieja} and so on.
This relative case includes the non-isolated case.
Incidentally, characterisations of $C^0$-sufficiency and $V$-sufficiency of weighted jets have been given by L. Paunescu (\cite{Paunescu1}, \cite{Paunescu2}).

The goal of this paper is to carry on the study of sufficiency of jets of 
differentiable map-germs $(\mathbb{R}^n,0)\to (\mathbb{R}^p,0)$, $n \ge p$, 
possibly with non-isolated singularities. 
We consider the following situation; for a given closed set-germ $\s$ 
in $( \mathbb{R}^n,0)$, we define the notion of map jets relative to $\s$, 
and using the group of homeomorphisms which fixes $\s$ pointwise, we define 
some topological sufficiencies of jets relative to $\s.$
In this paper we mainly treat the problems of $V$-sufficiency and 
$C^0$-sufficiency of relative jets.


Now we describe the plan of the rest of the paper.
The main results in this paper are the characterisations of $V$-sufficiency 
and $C^0$-sufficiency of jets relative to a given closed set.
Therefore let $\s$ be a germ of a closed set at $0 \in \R^n$ such that 
$0 \in \s,$ as above.

In \S 2 we first introduce the notion of jet relative to $\s.$
In the non-relative case, namely in the case where $\s=\{ 0\}$
any $r$-jet, $r \in \N$, has a unique polynomial representative of degree
not exceeding $r$. 
Therefore an $r$-jet can be identified with such a polynomial representative. 
But in the relative case some jets do not have even an analytic realisation 
(cf. Remark \ref{realisation}(1)). 
Therefore, when we consider the problem of sufficiency of relative jets 
in the general setting, we cannot use the analytic method like 
the curve selection lemma.
Note that the non-relative case is a special case of the analytic setting, 
namely $\s$ is a subanalytic closed set-germ and any relative $r$-jet 
has a subanalytic $C^r$ or $C^{r+1}$ realisation 
(c.f. H. Hironaka \cite{hironaka} for subanalyticity).
Subsequently we define the notions of $\s$-$C^0$-sufficiency, 
$\s$-$SV$-sufficiency and $\s$-$V$-sufficiency of relative jets,
and give the formulations of their criteria in the relative case.
Then we explain the role of the Bochnak-Lojasiewicz inequality 
(\cite{bochnaklojasiewicz}) in the non-relative case.
We also mention two important tools to characterise sufficiency of 
relative jets, integrability and the Bochnak-Kuo Lemma 
(\cite{bochnakkuo}).
At the end of this section, we introduce the Kuo quantity and the Thom 
quantity which have strong relationship with criteria for sufficiency
of jets.

The criteria for $C^0$-sufficiency of $r$-jets in $C^r$ functions and in
$C^{r+1}$ functions are the Kuiper Kuo condition and the second Kuiper-Kuo
condition, respectively, in the non-relative case.
J. Bochnak and W. Kucharz proved in \cite{bochnakkucharz} the corresponding
results in the mapping case to the above ones.
The main results of \S 3 are the generalisations of these results on 
$\s$-$C^0$-sufficiency of relative jets in $C^r$ mappings and $C^{r+1}$ 
mappings (Theorems \ref{RelativeKuiperKuo0}, \ref{RelativeKuiperKuo1}).

One of the main results of \S 4 is a characterisation of $\s$-$V$-sufficiency 
of relative $r$-jets in $C^r$ mappings : $(\R^n,0) \to (\R^p,0)$, $n \ge p$,
using the relative Kuo condition.
We give such a characterisation in the general setting (Theorem 
\ref{RelativeKuoThm1}) when $n > p$, and give it in the analytic setting 
when $n = p$ (Theorem \ref{RelativeKuoThm2}).
Using this result and the criterion of $\s$-$C^0$-sufficiency in $C^r$ mappings
in \S 3, we construct an example of a polynomial function to show that
$\s$-$V$-sufficiency in $C^r$ functions does not always imply 
$\s$-$C^0$-sufficiency in $C^r$ functions 
(Example \ref{nonequivalentexample1}).
From this example, we can see that the Bochnak-Lojasiewicz inequality 
does not always hold in the relative case even if $f$ is a polynomial 
function and $\s$ is a line (see Remark \ref{rmkBochnakLojasiewicz} also).
Another main result of \S 4 is a sufficient condition for the relative
$r$-jets to be $\s$-$V$-sufficient in $C^{r+1}$ mappings 
(Theorem \ref{RelativeKuoThm3}).
Using this result and the criterion of $\s$-$C^0$-sufficiency in $C^{r+1}$ 
mappings in \S 3, we construct an example of a polynomial function to show 
that $\s$-$V$-sufficiency in $C^{r+1}$ functions does not always imply 
$\s$-$C^0$-sufficiency in $C^{r+1}$ functions 
(Example \ref{nonequivalentexample3}).
As mentioned in the Abstract, $C^0$-sufficiency of $r$-jets in $C^r$ 
functions (resp. $C^{r+1}$ functions) is equivalent to $V$-sufficiency 
in $C^r$ functions (resp. $C^{r+1} $functions) in the non-relative case. 
But these equivalences do not always hold in the relative case. 
In this sense, Examples \ref{nonequivalentexample1} and 
\ref{nonequivalentexample3} are interesting applications of our main 
results.

In \cite{bochnakkuo} J. Bochnak and T.-C. Kuo gave characterisations of
finite $V$-determinacy for $C^{\infty}$ map-germs using $C^r$-rigidity
and ellipticity of some ideal.
Under the assumption that $\s$ is coherent (see Definition \ref{coherent}),
we generalise the Bochnak-Kuo theorem to the relative case
(Theorems \ref{T1}, \ref {T2}) in \S 5.

In \S 6 we deal with the contact equivalence in the relative case, 
generalising the main result of H. Brodersen \cite{brodersen}.
We study the relative contact equivalence, and show the equivalence 
for $C^{\infty}$ map-germs of infinite $\s$-contact determinacy 
and finite $\s$-contact determinacy ({Theorem \ref{thmequiv}).

\vspace{2mm}

Throughout this paper, let us denote by $\mathbb{N}$ the set of natural 
numbers in the sense of positive integers.


\section{Preliminaries}


\subsection{Definitions}   
 
 Let ${\mathcal E}_{[s]}(n,p)$ denote the set of
$C^s$ map-germs : $(\R^n,0)\to (\R^p,0)$, let $j^r f(0)$ denote the r-jet of 
$f$ at $0 \in \R^n$ for $f \in {\mathcal E}_{[s]}(n,p)$ ($s \ge r$),
and let $J^r(n,p)$ denote the set of r-jets in ${\mathcal E}_{[s]}(n,p)$.

Throughout this paper, let $\s$ be a germ of a closed subset of $ \R^n$ 
at $0 \in \R^n$ such that $0 \in \s.$  
Then we denote by $\mathcal{R}_\s^{\text{fix}}$ the group of germs of homeomorphisms  
$\vp:(\R^n,0) \to (\R^n,0) $ at $0 \in \R^n$ which fixes $\s$, namely
$\vp(x) = x$ for all $x\in \s.$ 
Finally we denote by $d(x,\s)$ the distance from a point $x \in \R^n$ 
to the subset $\s.$
 
 We consider on ${\mathcal E}_{[s]}(n,p)$ the following equivalence relation: 

\vspace{1mm}

\noindent Two map-germs  $f,g\,\in {\mathcal E}_{[s]}(n,p)$ are 
$r$-$\Sigma$-{\em equivalent}, denoted by $f\sim g$, if there exists 
a neighbourhood $U$ of $0$ in $\R^n$ such that the r-jet extensions of 
$f$ and $g$ satisfy $j^rf(\s\cap U)= j^rg(\s\cap U).$

\vspace{1mm}

\noindent We denote by $j^rf(\s;0)$ (or simply $j^rf(\s)$) the equivalence 
class of $f,$ and by $J^r_{\s}(n,p)$ the quotient set 
${\mathcal E}_{[s]}(n,p)/\sim.$ 

\begin{rmk}\label{realisation}
(1) In the case where $\s = \{ 0 \}$, an $r$-jet $j^rf(0)$ has 
a polynomial realisation for any $f \in {\mathcal E}_{[r]}(n,p)$.
But this property does not always hold in the relative case.
In fact, let $f : (\R,0) \to (\R,0)$ be a $C^{\infty}$ function 
defined by

\vspace{3mm}

$f(x) := \begin{cases}
e^{-\frac{1}{x^2}}\sin \frac{1}{x} & \mathrm{ if }\ x \ne 0\\
0 &  \mathrm{ if }\ x = 0.
\end{cases}$

\vspace{3mm}

\noindent Let $\s = \{ \frac{2}{m\pi} \ | \ m \in \mathbb{N}\} 
\cup \{ 0 \}$.
Then $f(\frac{2}{m\pi}) = 0$ for even $m$, but 
 $f(\frac{2}{m\pi}) \ne 0$ for odd $m$.
Therefore, for any $r \in \mathbb{N}$, $j^rf(\s ; 0)$ does not
have even a subanalytic $C^r$-realisation.

(2) Let $f \in {\mathcal E}_{[r]}(n,p)$, and let $\s$ be a germ 
of a closed subset of $\R^n$ at $0 \in \R^n$ such that $0 \in \s$.
Then $j^rf(\s ; 0)$ has a $C^r$-realisation $\tilde{f}$ whose
restriction to $\R^n \setminus \s$ is smooth, namely of class
$C^{\infty}$ (Theorem 2.2, page 73 in J.-C. Tougeron \cite{tougeron}).
\end{rmk}
Let us introduce some equivalences for elements of ${\mathcal E}_{[s]}(n,p)$.
\begin{defn} 
(1) We say that $f,g\,\in {\mathcal E}_{[s]}(n,p)$ are 
$\s$-$C^0$-{\em equivalent}, if there is $\vp \in \mathcal{R}_\s^{\text{fix}}$ 
such that $f=g\circ \vp$. 

(2) We say that  $f,g\,\in {\mathcal E}_{[s]}(n,p)$ are 
$\s$-V-{\em equivalent}, if $f^{-1}(0)$ is homeomorphic to $g^{-1}(0)$
as germs at $0\in \mathbb{R}^n$ by a homeomorphism which fixes 
$f^{-1}(0)\cap \s$.

(3) We say that $f,g\,\in {\mathcal E}_{[s]}(n,p)$ are 
$\s$-SV-{\em equivalent}, if there is a local homeomorphism 
$\vp \in \mathcal{R}_\s^{\text{fix}}$ such that $\vp (f^{-1}(0))=g^{-1}(0)$. 
\end{defn} 

Let $w \in J^r_{\s}(n,p).$ 
We call the relative jet $w$ $\s$-$C^0$-{\em sufficient}, 
$\s$-V-{\em sufficient,} and $\s$-SV-{\em sufficient} in 
${\mathcal E}_{[s]}(n,p)$ $(s\geq r)$, if any two realisations
$f$, $g\,\in {\mathcal E}_{[s]}(n,p)$ of $w,$ 
namely $j^rf(\s;0) = j^rg(\s;0)=w,$  are $\s$-$C^0$-equivalent, 
$\s$-V-equivalent, and $\s$-SV-equivalent, respectively.


We prepare some notations.

\begin{defn}
Let  $f,g : U \to \R$ be non-negative functions,
where $U \subset \R^N$ is an open neighbourhood of $0 \in \R^N$.
If there are real numbers $K > 0$, $\delta > 0$ 
with $B_{\delta}(0) \subset U$ such that
$$
f(x ) \le K g(x ) \ \ \text{for any} \ \ x \in B_{\delta}(0),
$$
where $B_{\delta}(0)$ is a closed ball in $\R^N$ of radius $\delta$
centred at $0 \in \R^N$, 
then we write $f \precsim g$ (or $g \succsim f$).
If $f \precsim g$ and $f \succsim g$, we write $f \thickapprox g$.
\end{defn}

The following lemma is useful in establishing many of the results in this paper.

\begin{lem}\label{lemrflat}
Let $\s$ be a germ at $0\in \mathbb{R}^n$ of a closed subset, and let
$f: ( \mathbb{R}^n,0)\to (\mathbb{R}^p,0)$ be a $C^k$ map-germ, $k\geq 1,$
such that $j^kf(\s;0)=\{0\}.$ Then
$
\|f(x)\|=o(d(x,\s)^k).
$
If  moreover $ f$ is of classe $ C^{k+1}$, then 
$
\|f(x)\| \precsim d(x,\s)^{k+1}.
$
\end{lem} 

\begin{proof}
It is a consequence of the Taylor formula for $C^k$ mapping and the 
assumption on $f.$ 
Let $\delta > 0$ and $y \in B(0,\delta ) \cap \s.$ 
By the $k$th order Taylor formula  we have
$$
f(x)=\sum_{j=0}^{k}\frac{(D^{j}f)(y)}{j!}((x-y)^{(j)})+R_{k,y}(x-y)
$$
for $ x\in B(0,r)$, where
$$
R_{k,y}(x-y)=\int_{0}^{1}\frac{(1-t)^{k-1}}{(k-1)!}((D^{k}f)(y+t(x-y))-(D^{k}f)(y))((x-y)^{(k)})\, dt
$$
satisfies
$
\Vert R_{k,y}(x-y)\Vert\leq C_{k,h,y}\Vert x-y\Vert^{k},\ \lim_{x-y\rightarrow 0}C_{k,x-y,y}=0
$
with,
$
C_{k,x-y,a}=\sup_{t\in[0,1]}\frac{\Vert(D^{k}f)(y+t(x-y))-(D^{k}f)(y)\Vert}{k!}.
$
The convergence $C_{k,h,y}\rightarrow 0$  as $h\rightarrow 0$  is uniform for $y$ supported in a compact subset of $B(0,\delta )$.
Where  $h^{(j)}= (h,\ \ldots,\ h)\in (\mathbb{R}^n)^{j}$ and 
$$
(D^{j}f)(y)(h,\ \ldots,\ h)=\sum_{(i_{1},\ldots,i_{j})}h_{i_{1}}\cdots h_{i_{j}}\frac{\partial^{k}f}{\partial x_{i_{1}}\cdots\partial x_{i_{j}}}(\alpha).
$$
 Now, if $j^kf(\s;0)=\{0\}$,   we have $
\dis \|f(x)\|\leq C_{k,h,y}\Vert x-y\Vert^{k}
$ and taking the infimum on $y\in \s$, we obtain $\|f(x)\|=o(d(x,\s)^k).$

Moreover, if $f$ is of classe $C^{k+1}$, taking the $(k+1)$th order Taylor formula, and the infimum on $y\in \s$, we get easily $
\|f(x)\| \precsim d(x,\s)^{k+1}.
$
\end{proof} 

\subsection{Relative Kuiper-Kuo condition and relative Kuo condition}

We suppose now on the germ $\s$ fixed, and introduce the relative 
notions to $\s$ of the Kuiper-Kuo condition and the Kuo condition.
We first give the notion of the relative Kuiper-Kuo condition.
The original condition was introduced by N. Kuiper \cite{kuiper}
and T.-C. Kuo \cite{kuo1} as a sufficient condition of
$C^0$-sufficiency of jets in the function case.

Let $v_1, \cdots , v_p$ be $p$ vectors in $\R^n$ where $n \ge p$.
The {\em Kuo distance $\kappa$} (\cite{kuo3}) is defined by
$\dis
\kappa(v_1, \ldots,v_p) = \displaystyle \min_{i}\{\text{distance of }\, 
v_i\, \text{ to }\, V_i\},
$
where $V_i$ is the span of the $ v_j$'s, $j\ne i$.
In the case where $p = 1$, $\kappa (v) = \| v \| .$

\begin{defn}[The relative Kuiper-Kuo condition]\label{RKK}  
A map germ $f\in {\mathcal E}_{[r]}(n,p)$, $n\geq p$, satisfies the 
{\it relative  Kuiper-Kuo condition $(K$-$K_{\s})$} if 
\begin{equation*}
\kappa(df(x))\succsim d(x,\s)^{r-1} 
\end{equation*} 
holds in some neighbourhood of $0 \in \R^n.$
\end{defn} 

\begin{defn}[The second relative Kuiper-Kuo condition]\label{SRKK}  
A map germ $f\in {\mathcal E}_{[r]}(n,p)$, $n\geq p$, satisfies the 
{\it second relative  Kuiper-Kuo condition $(K$-$K_{\s}^{\delta})$} 
if there is a strictly positive number $\delta$ such that
\begin{equation*}
\kappa(df(x))\succsim d(x,\s)^{r-\delta} 
\end{equation*} 
holds in some neighbourhood of $0 \in \R^n.$
\end{defn} 

For a map germ $f \in {\mathcal E}_{[r]}(n,p),$ we denote by $\sing (f)$ 
the singular points set of $f.$

\begin{rmk}\label{rmkRelKuiperKuo}
For a  map $f\in {\mathcal E}_{[r]}(n,p)$ satisfying the relative  Kuiper-Kuo 
condition or the second relative Kuiper-Kuo condition, 
we have $\sing(f) \subset \s$ in a neighbourhood of $0 \in \R^n.$
Therefore these conditions include the case where $\s = \sing (f)$,
as a special case.
\end{rmk} 

We next give the notion of the relative Kuo condition.
The original condition was introduced by T.-C. Kuo \cite{kuo3}
as a criterion of $V$-sufficiency of jets in the mapping case.

\begin{defn}[The relative  Kuo condition]\label{K}  
A map germ $f\in {\mathcal E}_{[r]}(n,p)$, $n\geq p$, satisfies the 
{\it relative  Kuo condition $(K_{\s})$} if there are strictly positive 
numbers $C, \alpha$ and $ \bar w$ such that
\begin{equation*}
\kappa(df(x))\geq Cd(x,\s)^{r-1} \text{ in } 
{\mathcal H}^{\s}_{r}(f; \bar w)\cap \{\| x \| < \alpha\},
\end{equation*} 
 $$\text{ namely, } \kappa(df(.))\succsim d(.,\s)^{r-1} \text{ on a set of points where }  \| f \| \precsim \ d(.,\s)^{r}.$$

\end{defn} 
In the definition \ref{K}, ${\mathcal H}^{\s}_{r}(f;\bar w)$ denotes the 
{\em horn-neighbourhood of $f^{-1}(0)$ relative to $\s$ of degree $r$ and 
width $\bar{w}$},
$
{\mathcal H}^{\s}_{r}(f;\bar w)=\{x\in \mathbb{R}^n: \| f(x) \|
\leq\bar w\ d(x,\s)^{r}\}.
$\\
The original notion of this horn-neighbourhood was introduced 
in \cite{kuo2}, for $\s=\{0\}$.

We have also a variant of the previous condition:
\begin{defn}[The second relative Kuo condition]\label{Kd}  
A map germ $f\in {\mathcal E}_{[r+1]}(n,p)$, $n\geq p$, satisfies the 
{\it second relative  Kuo condition $(K_{\s}^\de)$} if for any 
map $g\in {\mathcal E}_{[r+1]}(n,p)$ 
satisfying $j^{r}g(\s;0)=j^{r}f(\s;0)$ there are numbers
$C, \alpha,\delta$ and $ \bar w$ (depending on $g$), such
that
\begin{equation*}
\kappa(df(x))\geq Cd(x,\s)^{r-\delta} \text{ in }  \mathcal{H}^{\s}_{r+1}(g;\bar w)\cap\{\|x\|<\a\},
\end{equation*} 
namely, $\kappa(df(.))\succsim d(.,\s)^{r-\delta} $
on a set of points where $\| g(.) \| \precsim \ d(.,\s)^{r+1}.$

\end{defn} 
\begin{rmk}\label{rmk2ndRelativeKuo}
\begin{enumerate}[(1)]
\item For a  map $f\in {\mathcal E}_{[r]}(n,p)$ satisfying the  relative  Kuo condition or the second relative  Kuo condition, in a neighbourhood of $0 \in \R^n,$ the intersection of the singular set of $f,$ $\sing (f),$ and the horn 
neighbourhood
${\mathcal H}^{\s}_{r}(f; \bar w)$ is contained in $\s$, namely 
$$
\sing(f)\cap {\mathcal H}^{\s}_{r}(f; \bar w)\subset \s.
$$
In particular, in a neighbourhood of $0 \in \R^n$,  
$\mathrm{grad}f_{1}(x),\ldots,\mathrm{grad}f_{p}(x) $ are linearly independent  on $f^{-1}(0)\setminus \s.$ 

\item  For a  map $f\in {\mathcal E}_{[r]}(n,p)$ satisfying  the second relative  Kuo,  we have  for any 
map $g\in {\mathcal E}_{[r+1]}(n,p)$ satisfying $j^{r}g(\s;0)=j^{r}f(\s;0)$, in a neighbourhood of $0 \in \R^n,$ the intersection of the singular set of $f,$ $\sing (f),$ and the horn neighbourhood
${\mathcal H}^{\s}_{r+1}(g; \bar w)$ is contained in $\s$, namely 
$
\sing(f)\cap {\mathcal H}^{\s}_{r+1}(g; \bar w)\subset \s.
$
Since   $\|(f-g)(x)\|\precsim d(x,\s)^{r+1},$ we have $f^{-1}(0)\subset {\mathcal H}^{\s}_{r+1}(g; \bar w) ,$
then, in a neighbourhood of $0 \in \R^n$,  
$\mathrm{grad}f_{1}(x),\ldots,\mathrm{grad}f_{p}(x) $ are linearly independent  on $f^{-1}(0)\setminus \s.$ 

\end{enumerate}

\end{rmk} 

\begin{defn}[Condition ($\widetilde{K}_{\s}$)]
A map germ $f\in {\mathcal E}_{[r]}(n,p)$, $n\geq p$, satisfies
{\it condition} ($\widetilde{K}_{\s}$) if 
\begin{equation*}
d(x,\s)\kappa(df(x))+\|f(x)\|\succsim d(x,\s)^{r} 
\end{equation*} 
holds in some neighbourhood of $0 \in \R^n.$
\end{defn}
\begin{rmk}\label{remark210}
\begin{enumerate}[(1)]
\item Condition ($\widetilde{K}_{\s}$) was introduced in \cite{bekkakoike1}, 
in the case $\s=\{0\}$, in the proof of the equivalence between 
$V$-sufficiency and $SV$-sufficiency.
\item It is easy to see that condition ($\widetilde{K}_{\s}$) and the relative 
Kuo condition ($K_{\s}$) are equivalent.
\item The relative Kuiper-Kuo condition $(K$-$K_{\s}),$ the relative Kuo 
condition $(K_{\s})$, and condition ($\widetilde{K}_{\s}$) are invariant 
under rotation.
\end{enumerate}
\end{rmk} 
\begin{defn}[Condition ($\widetilde{K}^\d_{\s}$)]
A map germ $f\in {\mathcal E}_{[r+1]}(n,p)$, $n\geq p$, satisfies
{\it condition ($\widetilde{K}_{\s}^\d$})
if for any 
map $g\in {\mathcal E}_{[r+1]}(n,p)$ 
satisfying $j^{r}g(\s;0)=j^{r}f(\s;0)$ there exists 
$\delta>0$  (depending on $g$), such
that
\begin{equation*}
d(x,\s)\kappa(df(x))+\|g(x)\|\succsim d(x,\s)^{r+1-\d} 
\end{equation*} 
holds in some neighbourhood of $0 \in \R^n.$

\end{defn}

\begin{rmk}\label{remark211}
\begin{enumerate}[(1)]
\item The second relative Kuiper-Kuo condition $(K$-$K_{\s}^{\delta}),$ 
the second relative Kuo condition ($K^\d_{\s}$), and condition 
($\widetilde{K}^\d_{\s}$) are invariant under rotation.
\item Condition ($\widetilde{K}^\d_{\s}$) can be equivalently written as:
for any 
map $g\in {\mathcal E}_{[r+1]}(n,p)$ 
satisfying $j^{r}g(\s;0)=j^{r}f(\s;0)$ there exists 
$\delta>0$  (depending on $g$), such
that
\begin{equation*}
d(x,\s)\kappa(dg(x))+\|g(x)\|\succsim d(x,\s)^{r+1-\d} 
\end{equation*} 
holds in some neighbourhood of $0 \in \R^n.$

\end{enumerate}

\end{rmk}


\subsection{Bochnak-Lojasiewicz inequality}

Let us explain the role of the Bochnak-Lojasiewicz 
inequality playing in the problem of sufficiency of jets 
in the non-relative, function case.
Let $f : (\R^n ,0) \to (\R,0)$ be a $C^r$ function germ.
The $r$-jet of $f$ at $0 \in \R^n$, $j^r f(0)$, has a unique polynomial
representative $z$ of degree not exceeding $r$.
We do not distinguish the $r$-jet $j^r f(0)$ and the polynomial 
representative $z$ here.

\vspace{3mm}

\noindent {\bf Kuiper-Kuo condition.} There is a strictly positive number
$C$ such that
\begin{equation*}
\| \grad z(x) \| \ge C \| x\|^{r-1} 
\end{equation*} 
holds in some neighbourhood of $0 \in \R^n.$

\vspace{3mm}

The Kuiper-Kuo condition is equivalent to the $C^0$-sufficiency of $z$
in $C^r$ functions (N. Kuiper \cite{kuiper}, T.-C. Kuo \cite{kuo1}, 
J. Bochnak and S. Lojasiewicz \cite{bochnaklojasiewicz}).

\vspace{3mm}

\noindent {\bf Kuo condition.} There are strictly positive 
numbers $C, \alpha$ and $ \bar w$ such that
\begin{equation*}
\| \grad z(x) \| \ge C \| x\|^{r-1} \text{ in } {\mathcal H}^{\s}_{r}(f; \bar w)\cap \{\| x \| < \alpha\}.
\end{equation*} 

The Kuo condition is equivalent to the $V$-sufficiency of $z$
in $C^r$ functions (T.-C. Kuo \cite{kuo3}).

\vspace{3mm}

\noindent {\bf Condition ($\widetilde{K}$).}
There is a strictly positive number $C$ such that
\begin{equation*}
\| x\| \| \grad z(x) \| + |f(x)| \ge C \| x\|^r
\end{equation*} 
holds in some neighbourhood of $0 \in \R^n.$

\vspace{3mm}

This condition is the Kuo condition in a different way. 
Now, we recall the Bochnak-Lojasiewicz inequality .

\vspace{3mm}

\noindent {\bf Bochnak-Lojasiewicz inequality.} Let $f : (\R^n,0) \to
(\R ,0)$ be a $C^{\omega}$ function germ, and let $0 < \theta < 1.$
Then 
\begin{equation*} 
\| x \| \| \grad f(x)\| \ge \theta |f(x)|
\end{equation*} 
holds in some neighbourhood of $0 \in \R^n.$

\vspace{3mm}

From this inequality, it follows that the Kuo condition, or in fact 
condition ($\widetilde{K}$) is equivalent 
to the Kuiper-Kuo condition in the analytic case.
Therefore we can see that $V$-sufficiency in $C^r$ functions is equivalent 
to $C^0$-sufficiency in $C^r$ functions.

The Kuiper-Kuo condition, the Kuo condition and condition ($\widetilde{K}$)
are $r$-compatible in the sense of \cite{bekkakoike1}.
Therefore we can replace $z$ with $f$ in those conditions.

In a similar way to the $C^r$ case, we can see that $V$-sufficiency in 
$C^{r+1}$ functions is equivalent to $C^0$-sufficiency in $C^{r+1}$ 
functions, using the Bochnak-Lojasiewicz inequality.


\subsection{Integrability}

The standard condition for proving local integrability of vector fields is the 
Lipschitz condition. 
We shall use a more general controllability condition which yields the 
Lipschitz equivalence or even $C^k$-equivalence as a special case 
(see for instance T.-C. Kuo \cite{kuo1}, \cite{kuo3}, N. Kuiper \cite{kuiper},
F. Takens \cite{takens}, for the isolated singular case and E. Looijenga \cite{looijenga}, J. Damon \cite{damon} for family of isolated singularities). 

Let $\s$ be a germ of closed set of $ \mathbb{R}^n$ such that $0\in \s.$ 
Let $G$ be a  germ of $C^1$ vector field on $ \mathbb{R}^n\times \mathbb{R}^m\setminus \s\times \mathbb{R}^m$ which satisfies the {\it relative
Lipschitz condition:} 
$
\displaystyle \Vert G(x,\ t)\Vert\leq C d( x, \s)
$
where $ d( x, \s)$ denotes the distance of the point $x$ to $\s.$

For a fixed vector $v \in \{0\}\times\mathbb{R}^{m},$ we define
$X(x,t) := \begin{cases}
G(x,t)+v & \mathrm{ if }\ x \notin \s\\
v &  \mathrm{ if }\ x\in \s.
\end{cases}$

\vspace{3mm}

\noindent Then we have the following proposition.

\begin{prop}\label{integral}
 For $G(x,t)$  satisfying  the preceding conditions, $X(x,t)$  is locally integrable in the sense that there are a neighbourhood $W$ of $(x_{0},t_{0})$  in $ \mathbb{R}^{n}\times \mathbb{R}^m$,  $\delta>0$, and a family of homeomorphisms $\phi_{s}(x,t)$ defined on $W$  for $|s|<\delta$ so that $\phi_{0}=id$ and for $(x,t,s)\in W \times(-\delta,\ \delta)$,
$\dis
\frac{\partial\phi_{s}}{\partial s}= X\circ\phi_{s}.
$
\end{prop} 

\begin{lem}\label{lemintegral}
Let $U$ be an open subset of $\R^n\setminus \s$, 
let $0 \in (a,\ b)$ and let $G: U\times (a,\ b)\rightarrow \mathbb{R}^{n}$ 
be a continuous mapping which satisfies
$$
\displaystyle \Vert G(x,\ t)\Vert\leq C d( x, \s)
$$  
for some $C>0$ and  $(x,\ t)\in U\times (a,\ b)$. 
Let $\varphi$ : $(\alpha,\ \beta)\rightarrow U$  be an integral solution of the system of differential equations $y'=G(y,t)$
with the initial condition $\varphi(0) = x_{0}$ where $x_{0}\in U$
and $0 \in (\alpha ,\ \beta )\subset (a,\ b)$.
Then we have
$\dis 
d(x_{0}, \s)e^{-C |t|}\leq d(\varphi(t), \s)\leq d(x_{0}, \s) e^{C|t|}
$
for $t \in (\alpha,\ \beta)$.
\end{lem}

\begin{proof}
Since $\Vert\varphi(t)\Vert>0$ for $t \in(\alpha,\ \beta)$, 
we can define the function $\rho$ : $(\alpha,\ \beta)\rightarrow \mathbb{R}$ 
by $\displaystyle \rho(t)=\frac{1}{2}\ln\Vert\varphi(t)\Vert^{2}$ 
for $t \in(\alpha,\ \beta)$.
This function is differentiable and
$$
\displaystyle \rho'(t)=\frac{\langle\varphi(t),\varphi'(t)\rangle}
{||\varphi(t)\Vert^{2}}=\frac{\langle\varphi(t),G(\varphi(t),t)\rangle}
{||\varphi(t)\Vert^{2}}
$$
for $t \in(\alpha,\ \beta)$.  
From the mean value theorem, for every $t \in(0,\ \beta)$ there exists 
$\theta\in(0,\ t)$ such that $\rho(t)-\rho(0)=\rho'(\theta)t$.
Then we have
$$
|\rho(t)-\rho(0)|\leq|\rho'(\theta)|t \displaystyle \leq\frac{\Vert\varphi(\theta)\Vert \Vert G(\varphi(\theta),\theta)\Vert}{\Vert\varphi(\theta)\Vert^{2}}t 
\displaystyle \leq Ct.
$$
Therefore for every $t \in(0,\ \beta)$,
$$
\displaystyle \rho(0)-Ct \displaystyle \leq\rho(t)\leq\rho(0)+Ct.
$$
The above inequalities hold also for $t =0$. 
This ends the proof of the lemma.
\end{proof}

\begin{proof}[ Proof of Proposition~\ref{integral}] 
Let  $\gamma(s),$  $|s|< \epsilon,$ be an integral curve of $X$ in 
$(\mathbb{R}^{n}\setminus\s)\times \mathbb{R}^m.$ 
Then, by  Lemma \ref{lemintegral}, $\gamma(s)$ stays within a compact subset 
of $(U\setminus\s)\times \mathbb{R}^{m}$ when $\gamma(0)$ does. 
Thus, together with $\gamma(\mathrm{s})=\gamma(0)+sv$ for 
$\gamma(0)\in \s\times \mathbb{R}^{m},$ we obtain a continuous flow 
$\phi_{\mathrm{s}},\ |s|<\epsilon$ for $\gamma(0)$ in a sufficiently small 
compact neighbourhood of $x_{0}$. 
Thus there is a compact neighbourhood $W$ of $(x_{0},t_{0})$ in 
$ \mathbb{R}^{n+m}$ and a positive number $\delta$ so that the integral curves 
$\gamma(s)$ of $X$ with $\gamma(0)\in W\setminus(\s\times\mathbb{R}^{m})$ 
are defined for $|s|\leqq\delta$ and belong to 
$W\setminus(\s\times\mathbb{R}^{m})$.  
Then, we define
$
\phi(x,t,s):W\times[-\delta,\ \delta]\rightarrow W
$
by $\phi(x,t,s)=\gamma(s)$ where $\gamma$ is the integral curve of $X$ with 
$\gamma(0)=(x,t)$ (if $(x,t)\in \s\times \mathbb{R}^{m},$ then 
$\ \gamma(s)=(x,t) +sv)$ ).
This flow has a continuous inverse (in a smaller neighbourhood) by the same 
argument applied to $-X$ and uniqueness. 
Thus, it is a parametrised family of local homeomorphisms.
\end{proof}


\subsection{Bochnak-Kuo Lemma}

J. Bochnak and T.-C. Kuo proved a lemma in \cite{bochnakkuo}
in order to show a characterisation of finite $V$-determinacy
of map-germs.
Using a similar argument to the lemma, we can show the following 
lemma which will be used to show characterisations of relative
$C^0$-sufficiency of jets and relative $V$-sufficiency of jets.

\begin{lem}\label{bochnakkuo1}
Let $\{u_{\nu}^{(1)},\ \ldots,\ u_{\nu}^{(p)}\}_{\nu\in \mathbb{N}}$ 
{\it be a sequence of} $p$-{\it tuples of vectors in} $\mathbb{R}^{n}$,
and let $s \in \mathbb{N}$.  
Suppose that there is a sequence of positive numbers 
$\alpha_{\nu},$ $ \alpha_{\nu}\rightarrow 0$ such that
$$
d\ (u_{\nu}^{(1)},\sum_{k=2}^p \mathbb{R} u_{\nu}^{(k)})
= o(\alpha_{\nu}^{s}).
$$
{\it Then we can find a sequence} 
$\{\lambda_{\nu}^{(2)},\ \ldots,\ \lambda_{\nu}^{(p)}\}_{\nu\in \mathbb{N}}$ 
{\it of} $(p-1)$-{\it tuples of vectors in} $\mathbb{R}^{n}$, 
{\it satisfying the following three conditions}:

(i)  $|\lambda_{\nu}^{(k)}|=o(\alpha_{\nu}^s),\ 2\leqq k\leqq p$;

(ii) {\it For each} $\nu\in \mathbb{N},\ u_{\nu}^{(2)}+\lambda_{\nu}^{(2)},\ 
\ldots$ , $u_{\nu}^{(p)}+\lambda_{\nu}^{(p)}$ {\it are linearly independent}; 

(iii)  For each $\nu\in \mathbb{N},\ u_{\nu}^{(1)}$ {\it belongs to the 
subspace spanned by} $u_{\nu}^{(k)}+\lambda_{\nu}^{(k)},\ 2\leqq k\leqq p$.
\end{lem}

\begin{rmk}\label{bochnakkuo2}
In the case of the original Bochnak-Kuo Lemma, we suppose that
$$
d\ (u_{\nu}^{(1)},\sum_{k=2}^p \mathbb{R} u_{\nu}^{(k)})
= o(\alpha_{\nu}^{s})
$$
for all $s \in \mathbb{N}$ not a given $s \in \mathbb{N}$.
Then statement (i) holds for all $s \in \mathbb{N}$.
\end{rmk}

The original Bochnak-Kuo Lemma will be used to show a characterisation
of finite $SV$-determinacy of map-germs in the relative case.


\subsection{Kuo quantity and Thom quantity}

Related to the problem of sufficiency of jets, let us introduce the Kuo 
quantity $K_m$ and the Thom quantity $T_m$.

\begin{defn}\label{KTquantity}
Let $f\in {\mathcal E}_{[r]}(n,p)$ $(n\geq p),$ and let $m\geq 1$ be an 
integer. Let us define two functions of the variable $x$:
\begin{equation}
K_{m}(f,x) := \| x \|^m \sum_{1\leq i_1<\ldots<i_{p}\leq n} 
\left| \det \left( \frac{D(f_1,\ldots, f_p)}{D(x_{i_1},\ldots,x_{i_{p}})}(x)
\right) \right|^m + \| f(x) \|^m
\end{equation}
\begin{equation}
T_{m}(f,x) := \sum_{1\leq i_1<\ldots<i_{p+1}\leq n} 
\left| \det \left(\frac{D(f_1,\ldots, 
f_p,\rho)}{D(x_{i_1},\ldots,x_{i_{p+1}})}(x)\right) \right|^m 
+ \| f(x) \|^m
\end{equation}
where $\rho(x)=\| x\|^2.$
Note that $T_{m}(f,x) = \| f(x) \|^m$ in the case where $n = p$.
\end{defn} 

Concerning these quantities, we have the following result.

\begin{thm}\label{equivKT} (Main Theorem in \cite{bekkakoike2}) 
Let  $f : ( \mathbb{R}^n,0)\to ( \mathbb{R}^p,0)$, $n \ge p$, be a 
$C^{\omega}$ map-germ.
Then for any $m \in \mathbb{N},$ 
$$
K_{m}(f,.)\thickapprox T_{m}(f,.).
$$
\end{thm}


\section{Relative $C^0$-sufficiency of jets}

Let us recall that $\s$ is a germ of a non-empty, closed subset at $0\in \R^n$
such that $0 \in \s.$
In this section we give criteria for $\s$-$C^0$-sufficiency of relative 
$r$-jets in $C^r$ mappings and in  $C^{r+1}$ mappings, and compute some 
examples on relative $C^0$-sufficiency of jets using the criteria.


\subsection{Relative $C^0$ sufficiency of $r$-jets in $C^r$ mappings}

In this subsection we give a criterion of $\s$-$C^0$-sufficiency of 
$r$-jets in $C^r$ mappings, using the relative Kuiper-Kuo condition.

In the following theorem,  the relative Kuiper-Kuo condition implies  $\s$-$C^0$-sufficiency, is proved also in \cite{Migus-Rodak-Spodzieja}, 
with a slighly different method from ours.

\begin{thm}\label{RelativeKuiperKuo0}
Let $r$ be a positive integer, and let 
$f \in {\mathcal E}_{[r]}(n,p)$ where $n \ge p$.
Then the following conditions are equivalent.
\begin{enumerate}[(1)]
\item $f$ satisfies the relative Kuiper-Kuo condition $(K$-$K_{\s})$, 
namely
\begin{equation*}
\kappa(df(x))\succsim d(x,\s)^{r-1} \text{ holds in some neighbourhood of  }
0 \in \R^n.\end{equation*} 

\item The relative $r$-jet $j^r f(\s;0)$ is $\s$-$C^0$-sufficient 
in ${\mathcal E}_{[r]}(n,p)$.
\end{enumerate}
\end{thm}  

\begin{proof} 
We first show the implication (1) $\implies$ (2).
In the case where $r = 1$,  $0 \in \R^n$ is a regular point of $f$.
Therefore the theorem follows from the Implicit Function Theorem.
 We may assume that $r \ge 2$ after this.

Let $g \in {\mathcal E}_{[r]}(n,p)$ be an arbitrary mapping such that 
$j^r g(\s ;0) = j^r f(\s ;0)$.
We define a $C^{r}$ mapping $h:(\mathbb{R}^n,0)\to (\mathbb{R}^p,0)$ by 
$h(x) := g(x) - f(x)$.
Then $j^rh(x)=0$ for any $x\in \s.$ 

Let $t_0$ be an arbitrary element of $I :=[0,1].$
Define $F(x,t) := f(x)+th(x)$ for $t\in I.$
Since  $j^rh=0$ on $\s$ near $0 \in \R^n$, by Lemma \ref{lemrflat},
$\| h(x)\| =o(d(x,\s)^{r}).$
Then there exists a small neighbourhood $T$ of $t_{0}$ in $I$ such that 
$$
\|F(x,t)-F(x,t_{0})\|= o(\ds^{r})
$$ for any $t \in T$.
Therefore there are $\bar{w}$, $\alpha > 0$ such that
$$
\nu (d_{x}F( x,t_{0})) = \nu (df(x)+t_{0} dh(x)) \geq 
\nu (df(x))-t_{0}\Vert dh(x)\Vert \geq \frac{C}{2}d(x,\s)^{r-1}
$$ 
in $  \{\| x\| < \alpha\}.$ 
Then there exists $C'>0$ such that
\begin{equation}\label{eq2-1}
 \kappa(d_{x}F(x,t))\geq C'\ds^{r-1}
\end{equation}
for $(x,t)\in
W:=\{\| x\| < \alpha\}\times T.$ 
Thus, for $(x,t)\in W\setminus \s\times T$ the vectors 
$\grad_{x}F_j(x,t)$ $(1\leq j\leq p)$ are linearly independent.
Let for $(x,t)\in W\setminus \s\times T$, $V_{x,t}$ be the subspace spanned by the $\{\grad_{x}F_1(x,t),\ldots,\grad_{x}F_p(x,t)\}.$ 

Let us consider now $\{N_1(x,t), \ldots N_p(x,t)\}$ the basis of $V_{x,t}$ constructed as follows:
$$N_j(x,t)=\grad_{x}F_j(x,t)-\tilde N_j(x,t) \qquad (1\leq j\leq p)$$ where $\tilde N_j(x,t)$ is the projection 
of $\grad_{x}f_j(x,t)$ to
the subspace $ V^{j}_{x,t}$ spanned by the $\grad_{x}F_k(x,t),$ $k\ne j.$
Hence for  $j\in\{1,\dots,p\}$,  $\| N_j(x,t)\|$ is the distance of $ \grad_{x}F_j(x,t)$ to $  V^j_{x,t} .$ 
From the above we get, for any $j\in\{1,\dots,p\}$ and $(x,t)\in W,$
\begin{equation*}
\| N_j(x,t)\|\geq \kappa(d_{x}F(x,t)) \geq C'\ds^{r-1}.
\end{equation*}

To trivialise the family of level sets, we use a version of the Kuo vector 
field \cite{kuo1},

\begin{equation*}\label{eq3-1}
X(x,t)= \begin{cases}\dis
\frac{\partial }{\partial t}+\sum_{j=1}^p h_{j}(x)\frac{N_j(x,t)}{\| N_j(x,t)\|^2} & \text{ if } (x,t) \in W\setminus \s\times T\\
\dis\quad \frac{\partial }{\partial t}&\text{ if }  (x,t) \in W \cap 
\s\times T.
\end{cases}
\end{equation*}
Since, 
\begin{equation*}\label{eq3-2}
\left\| \sum_{j=1}^p h_{j}(x)\frac{N_j(x,t)}{\| N_j(x,t)\|^2}\right\|  \precsim \frac{ \|h(x)\|}{\| N_j(x,t)\|}\precsim \ds\end{equation*}
by  Proposition \ref{integral}, the following system of differential equations:
\begin{equation}\label{eq3-3}
y'=X(y,t) \text{ is integrable. }
\end{equation}

Now for $(x,t)\in W$ define $\gamma_{(x,t)}$ to be the maximal 
solution of (\ref{eq3-3})  such that $\gamma_{(x,t)}(t)=x.$ 
Let $H_{{0}},\tilde{H}_{0}:W \rightarrow \{ \| x\| < \alpha \}$ be given 
by $H_{0}(x,\ t) := \gamma_{(x,t_{0})}(t)$ and 
$\tilde{H}_{0}(y,\ t) := \gamma_{(y,t)}(t_0).$
By Proposition \ref{integral}, the mappings $H_{0}$ and $\tilde{H}_{0}$ are 
continuous and by uniqueness of the solutions of (\ref{eq3-3})  
we have for any $(x,t)\in W$ 
$$
\tilde{H}_{0}(H_{0}(x,\ t),\ t) =x,\ H_{0}(x,\ t_{0})=x, \text{ and } 
H_{0}(\tilde{H}_{0}(y,\ t),t)=y, \, H_{0}(x,t)=x,
$$
and $F( \gamma_{(x,t_{0})}(t),t)=F(x,t_{0})$ for all $t\in T$,
namely we have
$
f(H_{0}(x,\ t ))+th(H_{0}(x,\ t))=F(x,t_{0}),
$
for $(x,t)\in W$ (since on $\s,$ $h\equiv 0$).  
In particular, for all $t,t'\in T$, the germs of $F(x,t)$ and $F(x,t')$
at $0 \in \R^n$ are $\s$-homeomorphic (i.e. by a homeomorphism in $\mathcal{R}_\s^{\text{fix}}$). 
Finally, by compactness of $[0,1],$ we obtain that the germs of maps
$f$ and $g$ at $0 \in \R^n$ are $\s$-homeomorphic. 
It follows that $j^rf(\s ;0)$ is $\s$-$C^0$-sufficient in 
${\mathcal E}_{[r]}(n,p).$

We next show the implication (2) $\implies$ (1) in the case where $p \ge 2$.
Let the relative $r$-jet $j^r f(\s ;0)$ be $\s$-$C^0$-sufficient
in ${\mathcal E}_{[r]}(n,p)$.
Suppose  by reductio ad absurdum, that $$\kappa(df(x))\succsim d(x,\s)^{r-1} $$
is not satisfied in a neighbourhood of the origin. 
One can then find a sequence $ (x_\nu)_{\nu\geq 1} $ of points of 
$\mathbb{R}^n \setminus \s$ converging to $0 \in \R^n$ such that 
\begin{equation}\label{kk4}
\kappa(df(x_\nu))=o(d(x_\nu,\s)^{r-1}).
\end{equation}  

Extracting a subsequence from $(x_\nu)_{\nu\geq 1}$ if necessary, 
one can assume that
$$
\Vert x_{\nu+1}\Vert<\frac{1}{3} d(x_\nu, \s) 
$$ 
(which implies, in particular, that $ d(x_\nu,\s) $ decreases),
and that condition \eqref{kk4} implies:
$$\delta_{\nu}=o(d(x_\nu,\s)^{r-1})$$
where 
$
\dis\delta_{\nu} := \kappa(df(x_\nu)) =
\dist(\mathrm{grad}f_{j}(x_\nu),\sum_{k \ne j} \mathbb{R} \ 
\mathrm{grad}f_{k}(x_\nu))
$
for some $j$, $1 \le j \le p$.\\
By Remark \ref{remark210}(3), we may assume $j = 1$ after this.

Now we apply Lemma \ref{bochnakkuo1}, with 
$u_{\nu}^{(k)}=\mathrm{grad}f_{k}(x_\nu)$, $\alpha_{\nu}=d(x_\nu,\s)$ 
and $s = r - 1$, to find for each $\nu\in \mathbb{N},$   $p-1$ vectors,  
$\lambda_{\nu}^{( 2)}, \ldots, \lambda_{\nu}^{( p)}\in \mathbb{R}^{n}$ 
such that:

\begin{enumerate}[$\bf (a)$]
\item
$\|\lambda_{\nu}^{(k)}\|=o(d(x_\nu,\s)^{r-1}),\ k=2,,\ldots,\ p;$
\item  $\mathrm{grad}f_{2}(x_{\nu})+\lambda_{\nu}^{( 2)},\ \ldots, \mathrm{grad}f_{p}(x_{\nu})+\lambda_{\nu}^{( p)}$ are linearly independant in  $\mathbb{R}^{n}$;
\item $\dis \mathrm{grad}f_{1}(x_{\nu})\in\sum_{k=2}^{p}\mathbb{R} \ 
(\mathrm{grad}f_{k}(x_{\nu})+\lambda_{\nu}^{(k)})$.
\end{enumerate}

Let $\psi: \mathbb{R}^n\rightarrow \mathbb{R}$ be a $C^{\infty}$ function 
such that $ \psi(t)=1$ in a neighbourhood of $0\in \mathbb{R}^n$ and 
$\psi(t)=0$ for $|t|\displaystyle \geqq\frac{1}{4}.$
We define a map-germ $\eta=(\eta_{1},\ldots,\eta_{p}) : (\R^n,0) \to (\R^p,0)$ 
by:
$$
\eta_{1}(x) = \epsilon_{\nu} \psi\left(\frac{x-x_\nu}{d(x_\nu,\s)}\right)\|x-x_\nu\|^{2},
$$
$$
\eta_{k}(x)=\psi\left(\frac{x-x_\nu}{d(x_\nu,\s)}\right)\langle\lambda_{\nu}^{(k)},(x\ -x_\nu)\rangle,\ k=2,\ \ldots,\ p,
$$
for $x\in B_{\nu}$ and $\eta(x)=0$ for 
$x\displaystyle \not\in\bigcup_{\nu=1}^{\infty}B_{\nu}$,
where 
$
B_{\nu}=\{x\in \mathbb{R}^{n} :\, \|x-x_{\nu}\| \leq
\frac{1}{4} d( x_{\nu},\s) \},
$
and $(\epsilon_{\nu})_{\nu \geq 1}$ is a sequence of real numbers,  for $\nu\in \mathbb{N}$,.
 
Since $|\psi(t)|$ is bounded in $\mathbb{R}^n,$  we have 
\begin{equation}\label{eq-d1}
|\eta_{1}(x)|
\precsim \epsilon_{\nu}d(x_{\nu},\s)^{2},
\end{equation}
\begin{equation}\label{eq-d2}
|\eta_{k}(x)| =o(d(x_\nu,\s)^{r}), \ \ k = 2, \cdots , p.
\end{equation}
Therefore, if we choose the sequence $(\epsilon_{\nu})_{\nu \geq 1}$
so that $\epsilon_{\nu}=o(d(x_\nu,\s)^{r})$), we have
$$
\eta(x)=o(d(x,\s)^r).
$$
It follows that $g=f-\eta$ is a $C^{r}$-realisation of $j^rf(\s ;0)$ and $g(x_{\nu})=f(x_{\nu})$, for any $\nu\in \mathbb{N}.$

By condition $\bf (b)$, there is a small neighbourhood $V_{\nu}$ of $x_{\nu}$ 
such that the set 
$$
M_{\nu}=\{x\in V_{\nu}:\, f_{k}(x)-\eta_{k}(x)=f_k(x_{\nu}),\ k=2,\ \ldots,p\}
$$ 
is a differentiable manifold of codimension $p-1.$

From condition $\bf (c),$ for each $\nu\in \mathbb{N},$  there are real numbers  $a_{2,\nu},\ldots, a_{p,\nu}$ such that,
$$
\mathrm{grad}f_{1}(x_{\nu})=\sum_{k=2}^{p}a_{k,\nu}(\mathrm{grad}f_{k}(x_{\nu})+\lambda_{\nu}^{( k)}).
$$
Choose now $\epsilon_{\nu}=o(d(x_\nu,\s)^{r})$ more finely such that $x_\nu$ 
is a non-degenerate critical point of the restriction to $M_{\nu}$ of
$$
h_{\nu}(x)=f_{1}(x)-\eta_{1}(x)+\sum_{k=2}^{p}a_{k,\nu}(\eta_{k}(x)-f_{k}(x)).
$$
\begin{eqnarray*} 
g^{-1}(g(x_{\nu}))\cap V_{\nu} & = & \{x\in V_{\nu}:\, 
g_{j}(x)=f_j(x_{\nu}), \, 1\le j\le p \} \\
& = & \{x\in M_{\nu}:h_{\nu}(x)=h_{\nu}(x_{\nu})\}.
\end{eqnarray*}
By the choice of  $\epsilon_{\nu}$, this set is the intersection of  the locus of a non-degenerate quadratic form $h_{\nu}^{-1}(h_{\nu}(x_{\nu}))$  with a codimension $p-1$ 
manifold $M_{\nu}.$ 
Then if it is a topological manifold, necessarily it is reduced to a point. 

Now if $n-p\ge1,$  $g^{-1}(x_{\nu})$ cannot be a topological manifold of codimension $p$ and
 if $n=p$, for $x\in M_{\nu}$, $g(x)=(h_{\nu}(x),\ 0),$ thus $g$ is not 
injective in any neighbourhood of $x_{\nu}$ , since (the quadratic form) 
$h_{\nu}$ restricted to the one dimensional manifold  ${M_{\nu}}$ is not 
injective, but this contradicts the following lemma:

 \begin{lem}\label{transv0} Let $j^rf(\s ;0)$ is $\s$-$C^0$-sufficient 
in $\mathcal{E}_{[r]}(n,p)$.
Then for all maps $\theta \in\mathcal{E}_{[r]}(n,p)$ such that 
$ j^{r}\theta (\s;0)=j^rf(\s ;0)$,  and for all sequence  $\{x_{m}\}_{m\in \mathbb{N}}$ in  $ \mathbb{R}^n\setminus \s,$ $x_{m}\rightarrow 0,$  
$ m\rightarrow\infty,$
there is a neighbourhood of  $x_{m},$  for sufficiently large $m$,   such that  $g^{-1}(g(x_{m}))$ is a  topological manifold of codimension $p$ 
$\mathrm{ if }\ n>p$, or $g$ is injective $\mathrm{ if }\ n=p.$
\end{lem} 

\begin{proof}
We need for this, the following fact, which is a consequence of Sard's theorem: 
Let $U\subset \mathbb{R}^{n},\ V\subset \mathbb{R}^{p}$ be open sets, let
$F : U \times V\rightarrow \mathbb{R}^{p}$ be a smooth map and let
$\{b_{m}\}_{m\in \mathbb{N}}$ be a sequence of points in the regular values 
of $F.$ 
Then the set
$ \mathcal{R}	=\{y\in V: \forall m\in \mathbb{N}, b_{m} \text{ is a regular value of the map }  F_{y}:U\ni x\rightarrow F(x,y)\}$ 
is residual in $V.$

By Remark \ref{realisation}(2), there exists a $C^r$-realisation $\tilde{f}$
of $j^rf(\s ;0)$ such that the restriction of $\tilde{f}$ to
$\R^n \setminus \s$ is smooth.

Let $h:\mathbb{R}^{n}\rightarrow \mathbb{R}$ be a smooth flat function 
such that $h^{-1}(0)=\s.$
We consider now, the map 
$F: (\mathbb{R}^{n}\times \mathbb{R}^{p},(0,0)) \rightarrow 
(\mathbb{R}^{p},0)$ defined by
$$
\ F_{j}(x,y)= \tilde{f}_{j}(x)+y_{j}h(x), \ \ j=1, \cdots , p.
$$  
The restriction of $\ F$ to 
$(\mathbb{R}^{n}\backslash \s)\times \mathbb{R}^{p}$ is a submersion
around $(0,0) \in \R^n \times \R^p.$\\
Let $(x_{m})_{m\in \mathbb{N}}$ be a sequence of points of $ \mathbb{R}^{n}\backslash \s$ which tends to $0,$ then  $(\tilde{f}(x_{m}))_{m\in \mathbb{N}}$ is a sequence of regular values of 
$F|_{(\mathbb{R}^{n}\backslash \s)\times \mathbb{R}^{p}}$, and by the quoted 
version of Sard theorem, there is $y_{0} \in \mathbb{R}^{p}$ such that  
$(\tilde{f}(x_{m}))_{m\in \mathbb{N}}$ is a sequence of regular values of 
$F_{y_{0}}|_{\mathbb{R}^{n}\backslash \s}$.
Let $g_{0}=F_{y_{0}}$. 
Since $h$ is flat on $\s$, for all $r\in \mathbb{N},$ 
$j^{r}g_{0}(\s;0)=j^r {f}(\s;0).$

Now, by $\s$-$C^0$-sufficiency of $j^rf(\s ;0)$, there is a germ of 
homeomorphism $\varphi : (\mathbb{R}^{n},0) \rightarrow (\mathbb{R}^{n},0)$ 
such that the restriction of $\varphi$ to $\s$ is the identity and  $g_{0}\circ \varphi=f$. 
Thus $f^{-1}(f(x_{m}))=\varphi^{-1}(g_{0}^{-1}(f(x_{m}))$  is a topological 
manifold of codimension $p$ in a neighbourhood of $x_{m}$ (for large $m$), 
because  $\varphi^{-1}$ is a homeomorphism and $g_{0}^{-1}(f(x_{i}))$ is a 
smooth submanifold of $\mathbb{R}^{n}\backslash \s$ of codimension $p$ and if 
$n=p$, $f$ is injective in a neighbourhood of $x_{m}$. 
Therefore, by  $\s$-$C^0$-sufficiency of $j^rf(\s ;0),$ any mapping 
$\theta \in\mathcal{E}_{[r]}(n,p)$ of classe $C^{r}$ such that 
$ j^{r}\theta (\s;0)=j^rf(\s ;0)$, shares these properties.
\end{proof}

\noindent Therefore we can see that $f$ satisfies the relative Kuiper-Kuo
condition $(K$-$K_{\s})$.

The implication (2) $\implies$ (1) in the function case, namely $p = 1$,
follows similarly to the above mapping case, but more simply.
In this case we have
$$
\delta_{\nu} := \kappa(df(x_\nu)) = \| \mathrm{grad}f(x_\nu)\|
= o(d(x_\nu,\s)^{r-1}).
$$
We do not need to apply the Bochnak-Kuo Lemma.
We take $\eta (x) = \eta_1 (x)$ as the same function, 
and consider $M_{\nu} = V_{\nu}$.
We do not define $h_{\nu}$.
Instead, $g^{-1}(g(x_{\nu}))\cap V_{\nu}$ takes the same role 
as $h_{\nu}^{-1}(h_{\nu}(x_{\nu}))$ in this case.
Then the remainder follows in the same way.
This completes the proof of the theorem.
\end{proof} 
\subsection{Relative $C^0$ sufficiency of $r$-jets in $C^{r+1}$ mappings}

In this subsection we give a criterion of $\s$-$C^0$-sufficiency of 
$r$-jets in $C^{r+1}$ mappings, using the second relative Kuiper-Kuo condition.

\begin{thm}\label{RelativeKuiperKuo1}
Let $r$ be a positive integer, and let 
$f \in {\mathcal E}_{[r+1]}(n,p)$ where $n \ge p$.
Then the following conditions are equivalent.
\begin{enumerate}[(1)]
\item $f$ satisfies the second relative Kuiper-Kuo condition 
$(K$-$K_{\s}^{\delta})$, namely there is a strictly positive number $\delta$
such that
\begin{equation*}
\kappa(df(x))\succsim d(x,\s)^{r-\delta} 
\end{equation*} 
holds in some neighbourhood of $0 \in \R^n.$
\item The relative $r$-jet $j^r f(\s;0)$ is $\s$-$C^0$-sufficient 
in ${\mathcal E}_{[r+1]}(n,p)$.
\end{enumerate}
\end{thm}  

\begin{proof}
This theorem is shown in the same way as Theorem \ref{RelativeKuiperKuo0}.

In the case of $r-\d$, the implication (1) $\implies$ (2) follows,  
after noticing that, by Lemma \ref{lemrflat}, if $F(x,t) := f(x)+th(x)$ 
for $t\in I=[0,1]$ and  $j^rh(\s;0)=0$, 
then $\| h(x)\| \precsim d(x,\s)^{r+1}$ and there exists a small neighbourhood 
$T$ of $t_{0}$ in $I$ such that 
$$
\|F(x,t)-F(x,t_{0})\| \precsim d(x,\s)^{r+1}
$$ 
for any $t \in T.$

On the other hand, we can show the implication (2) $\implies$ (1) 
in the same way as above, by replacing everywhere $r-1$ with $r-\d.$
\end{proof}


\subsection{$\s$-$C^0$-sufficiency of jets in the function case}

In this subsection we restate Theorems \ref{RelativeKuiperKuo0}, 
\ref{RelativeKuiperKuo1} in the function case.
Related to these results, we shall discuss in the next section if
the Bochnak-Lojasiewicz inequality holds in the relative case, and 
the relationship between the relative $C^0$-sufficiency of jets
and the relative $V$-sufficiency of jets through the relationship
between the relative Kuiper-Kuo condition and condition 
$(\widetilde{K}_{\s})$.

\begin{thm}\label{RelativeKuiperKuoThm}
(1)  Let $r$ be a positive integer, and let 
$f\in {\mathcal E}_{[r]}(n,1).$
Then the inequality
\begin{equation*}
\| \grad f(x) \| \succsim d(x,\s)^{r-1} 
\end{equation*} 
holds in some neighbourhood of $0 \in \R^n$ if and only if
the relative $r$-jet $j^r f(\s;0)$ is $\s$-$C^{0}$-sufficient 
in ${\mathcal E}_{[r]}(n,1).$  

(2) Let $r$ be a positive integer, and let $f\in {\mathcal E}_{[r+1]}(n,1).$
Then there is a strictly positive number $\delta$
such that the inequality
\begin{equation*}
\| \grad f(x) \| \succsim d(x,\s)^{r-\delta} 
\end{equation*} 
holds in some neighbourhood of $0 \in \R^n$ if and only if
the relative $r$-jet $j^r f(\s;0)$ is 
$\s$-$C^{0}$-sufficient in ${\mathcal E}_{[r+1]}(n,1)$.
\end{thm}

\begin{rmk}\label{remKuiperKuo}
X. Xu also has obtained in \cite{Xu} a result that the inequality in 
Theorem \ref{RelativeKuiperKuoThm} (1) implies 
$\s$-$C^{0}$-sufficiency in ${\mathcal E}_{[r]}(n,1).$
\end{rmk}

\begin{example}\label{examRelKK1}
Let $f : (\R^2,0) \to (\R,0)$ be a polynomial function defined by
$$
f(x,y) := x^3,
$$
and let $\s := \{ x = 0 \}$.
Then we can easily see that $d((x,y),\s ) = |x|$, and
$$
\| \mathrm{grad} f(x,y)\| \succsim |x|^2
$$
in a neighbourhood of $(0,0) \in \R^2$.
It follows from Theorem \ref{RelativeKuiperKuoThm}(1) that
$j^3f(\s ;0)$ is $\s$-$C^0$-sufficient in $\mathcal{E}_{[3]}(2,1)$.
\end{example}

\begin{example}\label{examRelKK2}
Let $f_m : (\R^2,0) \to (\R,0)$, $m \ge 3$, be a polynomial function 
(\cite{kuo2}) defined by
$
f_m (x,y) := x^3 - 3xy^m.
$
Then we have $\mathrm{grad} f_m (x,y) = 
(3(x^2 -  y^m), - 3mxy^{m-1})$.

(1) Let $\s := \{ (0,0) \}$.
Then we have $d((x,y),\s ) = \| (x,y)\|$, and
$$
\| \mathrm{grad} f_m (x,y)\| \succsim \| (x,y)\|^{\frac{3m}{2}-1}
$$
in a neighbourhood of $(0,0) \in \R^2$.
We can check the above inequality, dividing a neighbourhood 
of $0 \in \R^2$ into the following three regions:
$$
A := \{ 3|x^2 - y^m| \le |x|^2 \} , \ \
B := \{ 3|x|^2 \le |y|^m \} , \ \
\R^2 \setminus (A \cup B).
$$

By the Kuiper-Kuo theorem \cite{kuiper}, \cite{kuo1}, 
$j^{\frac{3m-1}{2}}f(0)$ is $C^0$-sufficient in 
$\mathcal{E}_{[\frac{3m+1}{2}]}(2,1)$ if $m$ is odd,
and $j^{\frac{3m}{2}}f(0)$ is $C^0$-sufficient in 
$\mathcal{E}_{[\frac{3m}{2}]}(2,1)$ if $m$ is even.

(2) Let $\s := \{ x = 0 \}$.
Then we can see that
\begin{equation}\label{Lojasiewiczinequality1}
\| \mathrm{grad} f_m (x,y)\| \succsim |x|^{3 - \frac{2}{m}}
\end{equation}
in a neighbourhood of $(0,0) \in \R^2$.

We can show (\ref{Lojasiewiczinequality1}) as follows.
Let $\lambda : (\R,0) \to (\R^2,0)$ be an arbitrary analytic arc on $\R^2$
passing through $(0,0) \in \R^2$, not identically zero, denoted by
$$
\lambda (t) = (a_k t^k + \cdots , b_s t^s + \cdots ).
$$

In the case where $\lambda (t) = ( 0, b_s t^s + \cdots )$, $b_s \neq 0$,
$\lambda$ is contained in $\s$, and $d((x,y), \s ) = |x| = 0$ on $\lambda$.
Therefore we have
$
\| \mathrm{grad} f_m (x,y)\| \succsim |x|
$
on $\lambda$.
Thus we may assume after this that $a_k \ne 0$.

In the case where $2k < ms$, we have
$
\| \mathrm{grad} f_m (x,y)\| \succsim |t|^{2k}, \ |x| \thickapprox |t|^k
$
on $\lambda$. Therefore we have
$
\| \mathrm{grad} f_m (x,y)\| \succsim |x|^2
$
on $\lambda$ near $(0,0) \in \R^2$.

In the case where $2k \ge ms$, we have
$$
\| \mathrm{grad} f_m (x,y)\| \ge |\frac{\partial f_m}{\partial y}(x,y)|  
\succsim |t|^k |t|^{(m-1)s} \succsim 
|t|^{k + \frac{2(m-1)}{m} k} = |t|^{(3 - \frac{2}{m})k}, \
\ |x| \thickapprox |t|^k
$$
on $\lambda$. Therefore we have
$
\| \mathrm{grad} f_m (x,y)\| \succsim |x|^{3 - \frac{2}{m}}
$
on $\lambda$ near $(0,0) \in \R^2$.

Thus we have (\ref{Lojasiewiczinequality1}) in a neighbourhood
of $(0,0) \in \R^2$.

Note that 
$\dis
\frac{\partial f_m}{\partial x}(t^k,t^s) \equiv 0, \ \ 
\left|\frac{\partial f_m}{\partial y}(t^k,t^s)\right| = 3m|t|^{(3-\frac{2}{m})k}
$
in the case where $2k = ms$.
Therefore it follows from Theorem \ref{RelativeKuiperKuoThm}(1), (2) that
$j^3f_m (\s ;0)$ is not $\s$-$C^0$-sufficient in $\mathcal{E}_{[3]}(2,1)$
but $\s$-$C^0$-sufficient in $\mathcal{E}_{[4]}(2,1)$ for any $m \ge 3.$

(3) Let $\s := \{ y = 0 \}$.
Then, using a similar computation to the above one, we can see that
$
\| \mathrm{grad} f_m (x,y)\| \succsim |y|^{\frac{3m}{2}-1}
$
in a neighbourhood of $(0,0) \in \R^2$.
It follows from Theorem \ref{RelativeKuiperKuoThm}(1), (2) that
$j^{\frac{3m-1}{2}}f(\s ;0)$ is $\s$-$C^0$-sufficient in 
$\mathcal{E}_{[\frac{3m+1}{2}]}(2,1)$ if $m$ is odd,
and $j^{\frac{3m}{2}}f(\s ;0)$ is $\s$-$C^0$-sufficient in 
$\mathcal{E}_{[\frac{3m}{2}]}(2,1)$ if $m$ is even.
\end{example}

\section{Relative $V$-sufficiency of jets}

\subsection{Relative $V$-sufficiency of $r$-jets in $C^r$ mappings}
In this subsection we discuss the relationship between the Kuo condition
and $V$-sufficiency of $r$-jets in $C^r$ mappings which are relative to 
the closed set $\s \subset \R^n$ such that $0 \in \s$.
\begin{thm}\label{RelativeKuoThm}
Let $r$ be a positive integer, and let
$f \in {\mathcal E}_{[r]}(n,p)$, $n \ge p$.
If $f$ satisfies condition $(\widetilde{K}_{\s})$,
then the relative $r$-jet, $j^r f(\s;0)$, is $\s$-$V$-sufficient 
in ${\mathcal E}_{[r]}(n,p)$.
\end{thm}  
\begin{proof}
Because of the same reason as the theorem above, we may assume
that $r \ge 2$.

Let $g \in {\mathcal E}_{[r]}(n,p)$ be an arbitrary mapping such that 
$j^r g(\s ;0) = j^r f(\s ;0)$.
We define a $C^r$ mapping $h:(\mathbb{R}^n,0)\to (\mathbb{R}^p,0)$ by 
$h(x) := g(x) - f(x)$.
Then $j^rh(x)=0$ for any $x\in \s.$ 
By Lemma \ref{lemrflat}, $\| h(x)\| = o(d(x,\s)^r).$
Let $F(x,t) := f(x)+th(x)$ for $t\in I=[0,1],$ and $t_0 \in I.$
Then there exists a small neighbourhood $T$ of $t_{0}$ in $I$ such that 
$$
\|F(x,t)-F(x,t_{0})\|\leq \bar w\ \ds^{r}
$$ for any $t \in T$.
Thus $F(x,t)=0$ is contained in 
$\left({\mathcal H}^{\s}_r(F(x,t_{0}); \bar w)\cap 
\{\| x\| < \alpha\}\right)\times T$, 
hence we will concentrate our attention to this set.

Moreover there are $\bar{w}$, $\alpha > 0$ such that
$$
\nu (d_{x}F( x,t_{0})) = \nu (df(x)+t_{0} dh(x)) \geq 
\nu (df(x))-t_{0}\Vert dh(x)\Vert \geq \frac{C}{2}d(x,\s)^{r-1}
$$ 
for $x\in {\mathcal H}^{\s}_{r}(f(x); \bar w)\cap \{\| x\| < \alpha\}.$ 
Then   there exists $C'>0$ such that
\begin{equation}\label{eq2-1}
\kappa(d_{x}F(x,t))\geq C'\ds^{r-1}
\end{equation}
for $(x,t)\in \left({\mathcal H}^{\s}_r(F(x,t_{0}); \bar w)\cap 
\{\| x\| < \alpha\}\right)\times T.$

Set  
$W := \left({\mathcal H}^{\s}_{r}(F(x,t_{0}); \bar w)\cap 
\{\| x\| < \alpha\}\right)\times T.$ 
Thus, for $(x,t)\in W\setminus \s\times T$ the vectors 
$\grad_{x}F_j(x,t)$ $(1\leq j\leq p)$ are linearly independent.
Let for $(x,t)\in W\setminus \s\times T$, $V_{x,t}$ be the subspace spanned by the $\{\grad_{x}F_1(x,t),\ldots,\grad_{x}F_p(x,t)\}.$ 

Let us consider now $\{N_1(x,t), \ldots N_p(x,t)\}$ the basis of $V_{x,t}$ 
constructed in the case of relative $C^0$-sufficiency of jets:
$\dis N_j(x,t)=\grad_{x}F_j(x,t)-\tilde N_j(x,t), \ (1\leq j\leq p),$ where $\tilde N_j(x,t)$ is the projection 
of $\grad_{x}f_j(x,t)$ to
the subspace $ V^{j}_{x,t}$ spanned by the $\grad_{x}F_k(x,t),$ $k\ne j.$
Then, for any $j\in\{1,\dots,p\}$ and $(x,t)\in W,$
\begin{equation*}
\| N_j(x,t)\|\geq \kappa(d_{x}F(x,t)) \geq C'\ds^{r-1}.
\end{equation*}

To trivialise the family of zero sets, we use a version of Kuo vector field
as in the proof of Theorem \ref{RelativeKuiperKuo0}.

\begin{equation*}\label{eq3-1}
X(x,t)= \begin{cases}\dis
\frac{\partial }{\partial t}+\sum_{j=1}^p h_{j}(x)\frac{N_j(x,t)}{\| N_j(x,t)\|^2} & \text{ if } (x,t) \in W\setminus \s\times T\\
\dis\quad \frac{\partial }{\partial t}& \text{ if } (x,t) \in 
W \cap \s\times T.
\end{cases}
\end{equation*}

Since, 
\begin{equation*}\label{eq3-2}
\left\| \sum_{j=1}^p h_{j}(x)\frac{N_j(x,t)}{\| N_j(x,t)\|^2}\right\|  \precsim \frac{ \|h(x)\|}{\| N_j(x,t)\|}\precsim \ds\end{equation*}
by  Proposition \ref{integral}, the following system of differential equations:
\begin{equation}\label{eq5-52}
y'=X(y,t).
\end{equation}
is integrable.
Now for $(x,t)\in W$ define $\gamma_{(x,t)}$ to be the maximal 
solution of (\ref{eq5-52})  such that $\gamma_{(x,t)}(t)=x.$ 

Let $H_{{0}}, \tilde{H}_{0}: W \rightarrow 
{\mathcal H}^{\s}_{r}(F(x,t_{0}); \bar w ) \cap \{ \| x \| < \alpha \}$
be given by $H_{0}(x,\ t) =\gamma_{(x,t_{0})}(t)$ and 
$\tilde{H}_{0}(y,\ t) =\gamma_{(y,t)}(t_{0}).$
By Proposition \ref{integral}, the mappings $H_{0}$ and $\tilde{H}_{0}$ 
are continuous and by uniqueness of the solutions of (\ref{eq5-52}),  
we have for any $(x,t)\in W$ 
$$
\tilde{H}_{0}(H_{0}(x,\ t),\ t) =x,\ H_{0}(x,\ t_{0})=x, \text{ and } 
H_{0}(\tilde{H}_{0}(y,\ t),t) = y, \, H_{0}(x,t)=x
$$
and $F( \gamma_{(x,t_{0})}(t),t)=F(x,t_{0})$ for all $t\in T.$
We have 
$
f(H_{0}(x,\ t ))+th(H_{0}(x,\ t))=F(x,t_{0}),
$
for $(x,t)\in W$ since on $\s,$ $h\equiv 0.$  
In particular, for all $t,t'\in T$, the germs of $F(x,t)=0$ and $F(x,t')=0$
at $0 \in \R^n$ are $\s$-homeomorphic.

Finally, using the same compactness argument as above, we obtain that 
the germs of zero-sets $f(x)=0$ and $g(x)=0$ at $0 \in \R^n$ are 
$\s$-homeomorphic. 
\end{proof}

\begin{thm}\label{RelativeKuoThm1}
Let $r$ be a positive integer, and let 
$f \in {\mathcal E}_{[r]}(n,p)$ where $n > p$. 
Then the following conditions are equivalent.
\begin{enumerate}[(1)]
\item $f$ satisfies the relative Kuo condition $(K_{\s})$.
\item $f$ satisfies condition $(\widetilde{K}_{\s})$.
\item The relative $r$-jet $j^r f(\s;0)$ is $\s$-$V$-sufficient 
in ${\mathcal E}_{[r]}(n,p)$.
\end{enumerate}
\end{thm}  

\begin{thm}\label{RelativeKuoThm2}
Let $r$ be a positive integer, and let 
$f \in {\mathcal E}_{[r]}(n,n)$.
Suppose that $j^rf(\s ;0)$ has a subanalytic $C^r$-realisation
and that $\s$ is a subanalytic closed subset of $\R^n$ 
such that $0 \in \s$. 
Then the following conditions are equivalent.
\begin{enumerate}[(1)]
\item $f$ satisfies the relative Kuo condition $(K_{\s})$.
\item $f$ satisfies condition $(\widetilde{K}_{\s})$.
\item The relative $r$-jet $j^r f(\s;0)$ is $\s$-$V$-sufficient 
in ${\mathcal E}_{[r]}(n,n)$.
\end{enumerate}
\end{thm}  

\begin{proof}[Proofs of Theorem \ref{RelativeKuoThm1} and Theorem 
\ref{RelativeKuoThm2}]

We first assume that $f \in {\mathcal E}_{[r]}(n,p)$, $n \ge p$,
and we do not necessarily assume that $j^rf(\s ;0)$ has a subanalytic 
$C^r$-realisation or $\s$ is subanalytic in the case where $n = p$.

As mentioned in Remark \ref{remark210}(2), conditions (1) and (2) are 
equivalent.
The implication (1) $\implies$ (3) follows from Theorem \ref{RelativeKuoThm}.
Therefore we shall show that condition (3) implies condition (2).
Namely, $\s$-$V$-sufficiency of jets implies condition
$(\widetilde{K}_{\s})$.

Let the relative $r$-jet $j^r f(\s ;0)$ be $\s$-$V$-sufficient
in ${\mathcal E}_{[r]}(n,p)$.
Suppose by reductio ad absurdum, that $(\widetilde{K}_{\s})$
is not satisfied. 
One can then find a sequence $ (x_\nu)_{\nu\geq 1} $ of points of 
$\mathbb{R}^n \setminus \s$ converging to $0 \in \R^n$ such that 
\begin{equation}\label{kk3}
d(x_\nu,\s)\kappa(df(x_\nu))+\|f(x_\nu)\|=o(d(x_\nu,\s)^{r}).
\end{equation}  

Extracting a subsequence from $(x_\nu)_{\nu\geq 1}$ if necessary, 
one can assume that
$$
\Vert x_{\nu+1}\Vert<\frac{1}{3} d(x_\nu, \s) 
$$ 
(which implies, in particular, that $ d(x_\nu,\s) $ decreases),
and that condition \eqref{kk3} implies:
\begin{enumerate}[1)]
\item $|f_{k}(x_\nu)|=o(d(x_\nu,\s)^{r})$, for all  $1\leqq k\leqq p$;
\item $\delta_{\nu}=o(d(x_\nu,\s)^{r-1})$
where 
$
\dis\delta_{\nu} := \kappa(df(x_\nu)) =
\dist(\mathrm{grad}f_{j}(x_\nu),\sum_{k \ne j} \mathbb{R} \ 
\mathrm{grad}f_{k}(x_\nu))
$\\
for some $j$, $1 \le j \le p$.
By Remark \ref{remark210}(3), we may assume $j = 1$ after this.
\end{enumerate}

Now we apply Lemma \ref{bochnakkuo1}, with 
$u_{\nu}^{(k)}=\mathrm{grad}f_{k}(x_\nu)$, $\alpha_{\nu}=d(x_\nu,\s)$ 
and $s = r - 1$, to find for each $\nu\in \mathbb{N},$ $p-1$ vectors, 
$\lambda_{\nu}^{( 2)}, \ldots, \lambda_{\nu}^{( p)}\in \mathbb{R}^{n}$ 
such that:

\begin{enumerate}[$\bf (a)$]
\item $\|\lambda_{\nu}^{(k)}\|=o(d(x_\nu,\s)^{r-1}),\ k=2,,\ldots,\ p;$
\item  $\mathrm{grad}f_{2}(x_{\nu})+\lambda_{\nu}^{( 2)},\ \ldots, \mathrm{grad}f_{p}(x_{\nu})+\lambda_{\nu}^{( p)}$ are linearly independant in  
$\mathbb{R}^{n}$;
\item $\dis \mathrm{grad}f_{1}(x_{\nu})\in\sum_{k=2}^{p}\mathbb{R} \ 
(\mathrm{grad}f_{k}(x_{\nu})+\lambda_{\nu}^{(k)})$.
\end{enumerate}

Let $\psi: \mathbb{R}^n\rightarrow \mathbb{R}$ be a $C^{\infty}$ function 
such that $ \psi(t)=1$ in a neighbourhood of $0\in \mathbb{R}^n$ and 
$\psi(t)=0$ for $|t|\displaystyle \geqq\frac{1}{4}.$
We define a map-germ $\eta=(\eta_{1},\ldots,\eta_{p}) : (\R^n,0) \to (\R^p,0)$ 
by:
$$
\eta_{1}(x)=\psi\left(\frac{x-x_\nu}{d(x_\nu,\s)}\right)(f_{1}(x_\nu)+\epsilon_{\nu}\|x-x_\nu\|^{2}),
$$
$$
\eta_{k}(x)=\psi\left(\frac{x-x_\nu}{d(x_\nu,\s)}\right)(f_{k}(x_\nu)-\langle\lambda_{\nu}^{(k)},(x\ -x_\nu)\rangle),\ k=2,\ \ldots,\ p,
$$
for $x\in B_{\nu}$ and $\eta(x)=0$ for 
$x\displaystyle \not\in\bigcup_{\nu=1}^{\infty}B_{\nu}$,
where 
$\dis
B_{\nu}=\{x\in \mathbb{R}^{n} :\, \|x-x_{\nu}\| \leq
\frac{1}{4} d( x_{\nu},\s) \},
$
and $(\epsilon_{\nu})_{\nu \geq 1}$ is a sequence of real numbers, for $\nu\in \mathbb{N}$.
 
Let $K>0$ such that $|\psi(t)|\leq K$ in $\mathbb{R}^n.$ 
Then we have 
\begin{equation}\label{eq-d1}
|\eta_{1}(x)| \leq K(|f_{1}(x_\nu)|+\epsilon_{\nu}\|x-x_\nu\|^{2})
\precsim o(d(x_\nu,\s)^{r})+\epsilon_{\nu}d(x_{\nu},\s)^{2},
\end{equation}
\begin{equation}\label{eq-d2}
|\eta_{k}(x)| \leq K(|f_{k}(x_\nu)|+\|\lambda_{\nu}^{(k)}\|\|x\ -x_\nu\|) 
\precsim  o(d(x_\nu,\s)^{r}), \ \ k = 2, \cdots , p.
\end{equation}
Therefore, if we take the sequence $(\epsilon_{\nu})_{\nu \geq 1}$
so that $\epsilon_{\nu}=o(d(x_\nu,\s)^{r})$, we  have
$$
\eta(x)=o(d(x,\s)^r).
$$
It follows that $g=f-\eta$ is a $C^{r}$-realisation of $j^rf(\s ;0).$

By condition $\bf (b)$, there is a small neighbourhood $V_{\nu}$ of $x_{\nu}$ 
such that the set 
$$
M_{\nu}=\{x\in V_{\nu}:\, f_{k}(x)-\eta_{k}(x)=0,\ k=2,\ \ldots,p\}
$$ 
is a differentiable manifold of codimension $p-1.$

From condition $\bf (c),$ for each $\nu\in \mathbb{N},$  there are real numbers  $a_{2,\nu},\ldots, a_{p,\nu}$ such that,
$$
\mathrm{grad}f_{1}(x_{\nu})=\sum_{k=2}^{p}a_{k,\nu}(\mathrm{grad}f_{k}(x_{\nu})+\lambda_{\nu}^{( k)}).
$$
Choose now $\epsilon_{\nu}=o(d(x_\nu,\s)^{r})$ more finely such that $x_\nu$ 
is a non-degenerate critical point of
$
h_{\nu}(x)=f_{1}(x)-\eta_{1}(x)+\sum_{k=2}^{p}a_{k,\nu}(\eta_{k}(x)-f_{k}(x)).
$\\
Then
\begin{eqnarray*}
g^{-1}(0)\cap V_{\nu} & = & \{x\in V_{\nu}:\, 
g_{1}(x)=g_{2}(x)=\ldots=g_{p}(x)=0\} \\
& = & \{x\in M_{\nu}:h_{\nu}(x)=0\}.
\end{eqnarray*}
By the choice of  $\epsilon_{\nu}$, this set is the intersection of  the locus of a non-degenerate quadratic form $h_{\nu}^{-1}(0)$  with a codimension $p-1$ 
manifold $M_{\nu}.$ 
Therefore, modifying the sequence $\epsilon_{\nu},$ 
$\epsilon_{\nu}=o(d(x_\nu,\s)^{r})$), if necessary,
in the case where $n > p,$ $g^{-1}(0)$ cannot be a topological 
manifold of codimension $p,$ around $x_{\nu}$, $\nu \in \mathbb{N}$.

By construction, the map-germ $g$ has the same $r$-jet as $f$ as mentioned 
above, and its zero set $g^{-1}(0)$ contains the sequence 
$(x_{\nu})_{\nu\in \mathbb{N}}$ which is not in $\s$.
Therefore the germ $g^{-1}(0) \setminus \s$ at $0 \in \R^n$ is not empty. 
Since $j^rf(\s ;0)$ is $\s$-$V$-sufficient, the germ $f^{-1}(0) \setminus \s$ 
at $0 \in \R^n$ is not empty, either. 

It follows from the above arguments that we have the following properties
for $f \in \mathcal{E}_{[r]}(n,p)$ under the assumption that 
$j^rf(\s ;0)$ is $\s$-$V$-sufficient in $\mathcal{E}_{[r]}(n,p)$, but $f$ 
does not satisfy condition $(\widetilde{K}_{\s})$:

\vspace{2mm}

\noindent (P1) The germ $f^{-1}(0) \setminus \s$ at $0 \in \R^n$ is not empty.

\vspace{1mm}

\noindent (P2) In the case where $n > p$, $j^rf(\s ;0)$ has a $C^r$-realisation $g \in \mathcal{E}_{[r]}(n,p)$ such that near each $x_\nu,$ 
$ g^{-1}(g(x_{\nu}))=g^{-1}(0)$ is not a topological manifold 
of codimension $p$.

\vspace{2mm}

On the other hand, we have the following lemma.

\begin{lem}\label{transv} Let $j^rf(\s ;0)$ is $\s$-$V$-sufficient 
in $\mathcal{E}_{[r]}(n,p)$.
Suppose that the germ $f^{-1}(0) \setminus \s$ at $0 \in \R^n$ is not empty.
Then there exists $g_0 \in \mathcal{E}_{[r]}(n,p)$ with  
$j^rg_0(\s ;0) = j^rf(\s ;0)$ such that the germ of 
$g_0^{-1}(0) \setminus \s$ at $0 \in \R^n$ is not empty and 
a smooth submanifold of $\R^n$ of codimension $p.$
\end{lem}
 
\begin{proof}
We need for this, the following fact, which is a consequence of 
Sard's theorem: 
Let $U\subset \mathbb{R}^{n},\ V\subset \mathbb{R}^{p}$ be open sets, 
let $F:U\times V \rightarrow \mathbb{R}^{p}$ be a smooth map.
and let $b$ be a regular values of $F.$ 
Then for almost all $y \in V$, $b$ is a regular value of a map
$F_{y}:U\ni x\mapsto F(x,y)$.

By Remark \ref{realisation}(2), there exists a $C^r$-realisation $\tilde{f}$
of $j^rf(\s ;0)$ such that the restriction of $\tilde{f}$ to
$\R^n \setminus \s$ is smooth.
Since the jet $j^rf(\s ;0)$ is $\s$-$V$-sufficient in $\mathcal{E}_{[r]}(n,p)$,
the germ of $\tilde{f}^{-1}(0) \setminus \s$ at $0 \in \R^n$ is not empty.

Let $h:\mathbb{R}^{n}\rightarrow \mathbb{R}$ be a smooth flat function 
such that $h^{-1}(0)=\s.$
We consider now, the map 
$F: (\mathbb{R}^{n}\times \mathbb{R}^{p},(0,0)) \rightarrow 
(\mathbb{R}^{p},0)$ defined by
$$
\ F_{j}(x,y)= \tilde{f}_{j}(x)+y_{j}h(x), \ \ j=1, \cdots , p.
$$  
The restriction of $\ F$ to 
$(\mathbb{R}^{n}\backslash \s)\times \mathbb{R}^{p}$ is a submersion
around $(0,0) \in \R^n \times \R^p.$
Since the germ of $\tilde{f}^{-1}(0) \setminus \s$ at $0 \in \R^n$ 
is not empty,
$(F|_{(\mathbb{R}^{n}\backslash \s)\times \mathbb{R}^p})^{-1}(0) \ne 
\emptyset$ as germs at $(0,0) \in \R^n \times \R^p.$
In addition, $0 \in \R^p$ is a regular value of 
$F|_{(\mathbb{R}^{n}\backslash \s)\times \mathbb{R}^{p}}$. 
Therefore, by the above fact, there is  $y_{0} \in \mathbb{R}^{p}$ close to
$0 \in \R^p$ such that $0 \in \mathbb{R}^p$ is a regular value of 
$F_{y_{0}}|_{\mathbb{R}^{n}\backslash \s}$.

Now we let $g_0 = F_{y_{0}}$.
By construction, $g_0 \in \mathcal{E}_{[r]}(n,p)$, and 
$j^r g_0 (\s ;0) = j^rf(\s ;0).$
Since $j^rf(\s ;0)$ is $\s$-$V$-sufficient in $\mathcal{E}_{[r]}(n,p)$,
$g_{0}^{-1}(0) \setminus \s$ is not empty and a germ of a smooth submanifold 
of $\R^n$ of codimension $p.$ 
\end{proof}

We first consider the case where $n > p$.
Since $j^rf(\s ;0)$ is $\s$-$V$-sufficient in $\mathcal{E}_{[r]}(n,p)$, 
property $(P2)$ contradicts Lemma \ref{transv}.
Therefore $f$ satisfies condition $(\widetilde{K}_{\s})$.
This completes the proof of Theorem \ref{RelativeKuoThm1}.

We next consider the case where $n = p$.
In this case we are assuming that $j^rf(\s ;0)$ has a subanalytic 
$C^r$-realisation $\overline{f}$, and that $\s$ is a subanalytic subset 
of $\R^n$.
By property $(P1)$, the germ of $f^{-1}(0) \setminus \s$ at $0 \in \R^n$ 
is not empty.
Since $j^rf(\s ;0)$ is $\s$-$V$-sufficient in $\mathcal{E}_{[r]}(n,p)$, 
the germ of $\overline{f}^{-1}(0) \setminus \s$ at $0 \in \R^n$ 
is not empty.
Then, by the Curve Selection Lemma, there exists a $C^{\omega}$ arc 
$\lambda : [0,\delta) \to \R^n$, $\delta > 0$, such that 
$\lambda (0) = 0 \in \R^n$ and $\lambda (t) \in \overline{f}^{-1}(0)$, 
$t \in [0,\delta)$.
Therefore, because of $\s$-$V$-sufficiency of $j^rf(\s ;0)$, 
this contradicts Lemma \ref{transv}.
Therefore $f$ satisfies condition $(\widetilde{K}_{\s})$.
This completes the proof of Theorem \ref{RelativeKuoThm2}.
\end{proof}
\begin{rmk}\label{rmkBochnakLojasiewicz}
In the non-relative case $C^0$-sufficiency of $r$-jets in 
${\mathcal E}_{[r]}(n,1)$ is equivalent to $V$-sufficiency of 
$r$-jets in ${\mathcal E}_{[r]}(n,1)$.
The Bochnak-Lojasiewicz inequality takes a very important role
in the proof of the equivalence.
Therefore it may be natural to ask whether the Bochnak-Lojasiewicz
inequality holds also in the relative case.
More precisely, if we let $f : (\R^n,0) \to (\R,0)$ a $C^{\omega}$ 
function germ, then we ask whether the following inequality
$$
d(x,\s ) \| \mathrm{grad} f(x)\| \succsim |f(x)|
$$
holds in a neighbourhood of $0 \in \R^n$.

If this Bochnak-Lojasiewicz inequality holds in the relative case, 
then it follows that the relative Kuiper-Kuo condition ($K$-$K_{\s}$) 
and condition ($\widetilde{K}_{\s}$) are equivalent
like in the non-relative case.
But we give an example below to show that conditions ($K$-$K_{\s}$) and  
($\widetilde{K}_{\s}$) are not necessarily equivalent in the relative case.
As a result, we can see that the Bochnak-Lojasiewicz inequality does not 
always hold in the relative case, and it follows from Theorems 
\ref{RelativeKuiperKuo0}, \ref{RelativeKuoThm1} that $\s$-$V$-sufficiency 
of $r$-jets in ${\mathcal E}_{[r]}(n,1)$ does not always imply 
$\s$-$C^0$-sufficiency of $r$-jets in ${\mathcal E}_{[r]}(n,1)$.
\end{rmk}

\begin{example}\label{nonequivalentexample1}
Let us recall the situation in Example \ref{examRelKK2}(2).
Namely, $f_m (x,y) = x^3 - 3xy^m$, $m \ge 3$, and $\s = \{ x = 0 \}$.
Let $r = 3$.
In this setting, the relative Kuiper-Kuo condition is
$$
\| \mathrm{grad} f_m (x,y) \| \succsim |x|^{3-1}
$$
in a neighbourhood of $(0,0) \in \R^2$.
But as seen in Example \ref{examRelKK2}(2), the above inequality
does not hold along an analytic arc $\lambda (t) = (t^m,t^2)$
for $m \ge 3.$
In other words, the relative Kuiper-Kuo condition ($K$-$K_{\s}$)
is not satisfied.
Therefore, by Theorem \ref{RelativeKuiperKuo0}, $j^3 f_m (\s ; 0)$
is not $\s$-$C^0$-sufficient in ${\mathcal E}_{[3]}(2,1)$.

On the other hand, condition ($\widetilde{K}_{\s}$) is
$$
|x| \| \mathrm{grad} f_m (x,y)\| + |f_m (x,y)| \succsim |x|^3 
$$
in a neighbourhood of $(0,0) \in \R^2$.
We show that $f_m$, $m \ge 3$, satisfies this condition.
Let
$$
\lambda (t) = (a_k t^k + \cdots , b_s t^s + \cdots )
$$
be an analytic arc passing through $(0,0) \in \R^2$ as in 
Example \ref{examRelKK2}.
Then we may assume $a_k \ne 0$, and $|x| \thickapprox |t|^k.$

In the case where $2k < ms$, we have
$$
|x| \| \mathrm{grad} f_m (x,y)\| \ge 3|x||x^2 - y^m| \ge |x||x|^2
\thickapprox |t|^k|t|^{2k} = |t|^{3k}
$$
on $\lambda$ near $(0,0) \in \R^2$.

In the case where $2k > ms$, we have
$$
|x| \| \mathrm{grad} f_m (x,y)\| \ge 3|x||x^2 - y^m| \ge |x||y|^m
\succsim |t|^k|t|^{2k} = |t|^{3k}
$$
on $\lambda$ near $(0,0) \in \R^2$.

In the case where $2k = ms$ and $a_k \ne b_s$, we have
$$
|x| \| \mathrm{grad} f_m (x,y)\| \ge 3|x||x^2 - y^m| \succsim |x| |x|^2
\thickapprox |t|^k|t|^{2k} = |t|^{3k}
$$
on $\lambda$ near $(0,0) \in \R^2$.

In the case where $2k = ms$ and $a_k = b_s$, we have
$$
|f_m (x,y)| = |x^3 - 3xy^m| = |x||x^2 - 3 y^m| \succsim |x| |x|^2
\thickapprox |t|^k|t|^{2k} = |t|^{3k}
$$
on $\lambda$ near $(0,0) \in \R^2$.

On any analytic arc $\lambda$, condition ($\widetilde{K}_{\s}$) is
satisfied.
Therefore we can see that $f_m$, $m \ge 3$, satisfies condition 
($\widetilde{K}_{\s}$).
It follows that conditions ($K$-$K_{\s}$) and ($\widetilde{K}_{\s}$) 
are not necessarily equivalent in the relative case.
In addition, by Theorem \ref{RelativeKuoThm1}, we see that 
$j^3f_m(\s ;0)$ is $\s$-$V$-sufficient in $\mathcal{E}_{[3]}(2,1)$ 
for any $m \ge 3.$

Incidentally, the Bochnak-Lojasiewicz inequality does not hold 
along an analytic arc $\lambda (t) = (t^m,t^2)$ for $m \ge 3$.
\end{example}

As a corollary of the proofs of Theorems \ref{RelativeKuoThm1} and 
\ref{RelativeKuoThm2}, we have the following.

\begin{cor}\label{manifold}
Let $r$ be a positive integer, and let $f \in {\mathcal E}_{[r]}(n,p)$ such that $j^rf(\s,0)$ is $\s$-$V$-sufficient in ${\mathcal E}_{[r]}(n,p).$

\begin{enumerate}[1)]
\item if $n>p,$  then  for any $C^{r}$ realisation $g$ of $j^rf(\s,0)$,  $g^{-1}(0)\setminus \s$ is a  germ of  $C^r$ submanifold of codimension $p$ at $0$ or empty.

\item if $n=p,$ $j^rf(\s ;0)$ has a subanalytic $C^r$-realisation
and  $\s$ is a germ at $0 \in \R^n$ of a closed subanalytic subset of $\R^n$, 
then for any $C^{r}$ realisation $g$ of $j^rf(\s,0)$, the set-germs 
$(g^{-1}(0),0)$ and $(f^{-1}(0),0)$ are equal and are contained in $(\s,0).$

\end{enumerate}
\end{cor} 
\begin{rmk}\label{importantrmk}
It is well-known that the Kuiper-Kuo condition and $V$-sufficiency of jets
are equivalent for function-germs.
But, by Example \ref{nonequivalentexample1} and Theorem \ref{RelativeKuoThm1}, 
we can see that they are not always equivalent in the relative case.
\end{rmk}

\begin{rmk}
It is worth to mention that if   $f\in \mathcal{E}_{r}(n,p)$ and a subanalytic mapping then it has a subanalytic realisation in $ \mathcal{E}_{q}(n,p)$ for any $q\ge r$ (see \cite{kurdykapawlucki}).
\end{rmk} 


\subsection{Relative V-sufficiency of $r$-jets in  $C^{r+1}$ mappings }
In this subsection we give some characterisations for the relative
$r$-jets to be $\s$-$V$-sufficient in $C^{r+1}$ mappings.

\begin{thm}\label{RelativeKuoThm3}
Let $r$ be a positive integer, and let
$f \in {\mathcal E}_{[r+1]}(n,p)$, $n \ge p$.
If $f$ satisfies condition $(K^\d_{\s})$,
then the relative $r$-jet, $j^r f(\s;0)$ is $\s$-$V$-sufficient 
in ${\mathcal E}_{[r+1]}(n,p)$.
\end{thm}  

\begin{proof}
Because of the same reason as the theorem above, we may assume
that $r \ge 2$.

Let $g \in {\mathcal E}_{[r+1]}(n,p)$ be an arbitrary mapping such that 
$j^r g(\s ;0) = j^r f(\s ;0)$.
We define a $C^{r+1}$ mapping $h:(\mathbb{R}^n,0)\to (\mathbb{R}^p,0)$ by 
$h(x) := g(x) - f(x)$.
Then $j^rh(x)=0$ for any $x\in \s.$ 

Let $F(x,t) := f(x)+th(x)$ for $t\in I=[0,1].$
Since  $j^rh=0$ on $\s$ near $0 \in \R^n$, by Lemma \ref{lemrflat},
$\| h(x)\| \precsim d(x,\s)^{r+1}.$
Then there exists a small neighbourhood $T$ of $t_{0}$ in $I$ such that 
$$
\|F(x,t)-F(x,t_{0})\|\leq \bar w\ \ds^{r+1}
$$ for any $t \in T$.
Thus the zero-set $F(x,t)=0$ is contained in 
$$
\left({\mathcal H}^{\s}_{r+1}(F(x,t_{0}); \bar w)\cap 
\{\| x\| < \alpha\}\right)\times T,
$$ 
hence we will concentrate our attention to this set.

Moreover there are $\bar{w}$, $\alpha > 0$ such that
$$
\nu (d_{x}F( x,t_{0})) = \nu (df(x)+t_{0} dh(x)) \geq 
\nu (df(x))-t_{0}\Vert dh(x)\Vert \geq \frac{C}{2}d(x,\s)^{r-\delta}
$$ 
in $ {\mathcal H}^{\s}_{r+1}(f; \bar w)\cap \{\| x\| < \alpha\}.$ 
Then there exists $C'>0$ such that
\begin{equation}\label{eq2-1}
 \kappa(d_{x}F(x,t))\geq C'\ds^{r-\d}
\end{equation}
for $(x,t)\in \left({\mathcal H}^{\s}_{r+1}(F(x,t_{0}); \bar w)\cap 
\{\| x\| < \alpha\}\right)\times T.$ 
Set  
$$
W := \left({\mathcal H}^{\s}_{r+1}(F(x,t_{0}); \bar w)\cap 
\{\| x\| < \alpha\}\right)\times T.
$$ 

Now we consider as in the proof of Theorem \ref{RelativeKuoThm} the basis  $\{N_1(x,t), \ldots N_p(x,t)\}$ of $V_{x,t}$ constructed as follows:
$$
N_j(x,t)=\grad_{x}F_j(x,t)-\tilde N_j(x,t), \ (1\leq j\leq p),
$$
where $\tilde N_j(x,t)$ is the projection 
of $\grad_{x}f_j(x,t)$ to
the subspace $ V^{j}_{x,t}$ spanned by the $\grad_{x}F_k(x,t),$ $k\ne j.$

From the above we get, for any $j\in\{1,\dots,p\}$ and $(x,t)\in W,$
\begin{equation*}
\| N_j(x,t)\|\geq \kappa(d_{x}F(x,t)) \geq C'\ds^{r-\d}.
\end{equation*}
and then we use the same vector field of trivialisation as above

\begin{equation*}\label{eq3-1}
X(x,t)= \begin{cases}\dis
\frac{\partial }{\partial t}+\sum_{j=1}^p h_{j}(x)\frac{N_j(x,t)}{\| N_j(x,t)\|^2} & \text{ if } (x,t) \in W \setminus \s \times T\\
\dis\quad \frac{\partial }{\partial t} & \text{ if }  (x,t) \in 
W \cap \s \times T.
\end{cases}
\end{equation*}

Since
\begin{equation*}\label{eq3-2}
\left\| \sum_{j=1}^p h_{j}(x)\frac{N_j(x,t)}{\| N_j(x,t)\|^2}\right\|  \precsim \frac{ \|h(x)\|}{\| N_j(x,t)\|}\precsim \ds^{1+\d},\end{equation*}
we use  Proposition \ref{integral} to end the proof as in Theorem \ref{RelativeKuoThm}.
\end{proof}

\begin{example}\label{nonequivalentexample2}
Let $f: (\R^n,0) \to (\R,0)$, $n \ge 3$, be a polynomial function defined by
$
f(x_{1},x_{2},\dots,x_{n}) := x_{1}^3 - 3x_{1}x_{2}^{5}
$
and $\s := \{ x_{1}=x_{2}=0 \}.$
Then we have 
$
\grad f_m (x,y) = (3(x_{1}^2 -  x_{2}^{5}), - 15x_{1}x_{2}^{4}),
$ 
and $d(x,\s ) = \| (x_{1},x_{2})\|.$
From the computation in Example \ref{examRelKK2}(1), 
$
\| \mathrm{grad} f (x)\| \succsim d(x,\s )^{7-\frac{1}2}
$
in a neighbourhood of $0 \in \R^n$.
Therefore, by Theorem \ref{RelativeKuiperKuoThm}(2), $j^{7}f(\s;0)$ is 
$\s$-$C^0$-sufficient in $\mathcal{E}_{[8]}(n,1).$
Now, since   $g(x)=x_{1}^3 - 3x_{1}x_{2}^{5}+x_{2}^{\frac{15}{2}}=(x_{1}-x_{2}^{ \frac{5}{2}})^2(x_{1}+2x_{2}^{ \frac{5}{2}})$ is a realisation of the jet $j^7f(\s;0)$ in $\mathcal{E}_{[7]}(n,1)$, which is not $\s$-$V$-equivalent
to $f;$ therefore $j^{7}f(\s;0)$ is not $\s$-$V$-sufficient in 
$\mathcal{E}_{[7]}(n,1).$ The proof can be carried out like in \cite{koku}.
\end{example}

By Lemma \ref{lemrflat}, we have the following as a corollary
of Theorem \ref{RelativeKuoThm3}.

\begin{cor}\label{RelativeKuoCor2}
Let $r$ be a positive integer, and let $f \in {\mathcal E}_{[r+1]}(n,p)$,
$n \ge p$.
If there exists $\delta>0$ such that
\begin{equation*}
d(x,\s)\kappa(df(x))+\|f(x)\|\succsim d(x,\s)^{r+1-\d}
\end{equation*} 
holds in some neighbourhood of $0 \in \R^n,$
then $j^rf(\s;0)$ is $\s$-$V$-sufficient in ${\mathcal E}_{[r+1]}(n,p).$
\end{cor}

\begin{rmk}\label{nonequivremark}
In the non-relative case $C^0$-sufficiency of $r$-jets in 
${\mathcal E}_{[r+1]}(n,1)$ is equivalent to $V$-sufficiency of 
$r$-jets in ${\mathcal E}_{[r+1]}(n,1)$, too.
But this does not holds in the relative case, namely
we give an example below to show that $\s$-$V$-sufficiency of $r$-jets 
in $C^{r+1}$ functions does not always imply $\s$-$C^0$-sufficiency of 
$r$-jets in $C^{r+1}$ functions, either.
\end{rmk}

\begin{example}\label{nonequivalentexample3}
Let $f : (\R^2,0) \to (\R,0)$ be a polynomial function defined by
$
f(x,y) := (x - y^3)^2 + y^{10},
$
and let  $\s = \{ x = 0 \}$.
Then we have 
$$
\mathrm{grad} f(x,y) = (2(x -  y^3), - 6y^2(x - y^3) + 10y^9).
$$
Let
$
\lambda (t) = (a_k t^k + \cdots , b_s t^s + \cdots )
$
be an analytic arc passing through $(0,0) \in \R^2$ as in 
Example \ref{examRelKK2}.
Then we may assume $a_k \ne 0$, and then $|x| \thickapprox |t|^k.$

In the case where $k < 3s$, we have
$$
\| \mathrm{grad} f(x,y)\| \ge 2|x - y^3| \ge |x|
$$
on $\lambda$ near $(0,0) \in \R^2$.

In the case where $k > 3s$, we have
$$
\| \mathrm{grad} f(x,y)\| \ge 2|x - y^3| \ge |y|^3 \ge |x|
$$
on $\lambda$ near $(0,0) \in \R^2$.

In the case where $k = 3s$, $|x| \thickapprox |y|^3.$
Therefore we have
$$
|f(x)| = (x - y^3)^2 + y^{10} \ge y^{10} \thickapprox |x|^{4-\frac{2}{3}}
$$
on $\lambda$ near $(0,0) \in \R^2$.

On any analytic arc $\lambda$, 
$\dis
|x| \| \mathrm{grad} f(x,y)\| + |f(x)| \succsim |x|^{4-\frac{2}{3}}
$
holds near $(0,0) \in \R^2$.
Therefore the above inequality holds in a neighbourhood of $(0,0) \in \R^2.$
It follows from Corollary \ref{RelativeKuoCor2} that
$j^3 f(\s ;0)$ is $\s$-$V$-sufficient in ${\mathcal E}_{[4]}(2,1).$

Let $\lambda (t) := (t^3,t).$
Then $|x| = |t|^3 = |y|$ on $\lambda.$
Therefore we have 
$$
\| \mathrm{grad} f(x,y)\| = \| (0,10t^9 )\| = 10 |t|^9 
= 10 |x|^3
$$
on $\lambda$ near $(0,0) \in \R^2 .$
By Theorem \ref{RelativeKuiperKuoThm}(2), $j^3 f(\s ;0)$ cannot 
be $\s$-$C^0$-sufficient in ${\mathcal E}_{[4]}(2,1).$
\end{example}

We gave a sufficient condition for the relative $r$-jets 
to be $\s$-$V$-sufficient in $C^{r+1}$ mappings.
We next give a necessary condition.

\begin{defn}
Let $f \in {\mathcal E}_{[r]}(n,p)$ and $d\in \mathbb{N}.$
We say that the horn neighbourhood of f, $\mathcal{H}^{\s}_{d}(f)$ 
is $\s$-{\em regular} if for some $w>0,$
$$
\kappa(df(x))\succsim d(x,\s)^{d-1} \text{ for $ x\in \mathcal{H}^{\s}_{d}(f;w),\ x$ near $0$.  }
$$

\end{defn} 

\begin{rmk}For germ $f\in {\mathcal E}_{[r]}(n,p)$, $n\geq p$, the following conditions 
are equivalent:
\begin{enumerate}[1)]
\item $ \mathcal{H}^{\s}_{r}(f)$ is $\s$-regular 
\item $f$ satisfies  condition ($\widetilde{K}_{\s}$).
\end{enumerate}
\end{rmk} 

\begin{prop}\label{RelativeKuoThm4}
Let $r$ be a positive integer, and let
$f \in {\mathcal E}_{[r+1]}(n,p)$, $n \ge p$, such that the relative $r$-jet $j^r f(\s;0)$ is $\s$-$V$-sufficient in ${\mathcal E}_{[r+1]}(n,p).$ 
Then for any realisation  $g$ of  $j^rf(\s ;0)$ in ${\mathcal E}_{[r+1]}(n,p)$, the horn neighbourhood $ \mathcal{H}^{\s}_{r+1}(g)$ is $\s$-regular.

\end{prop}  

\begin{proof} 
If not, then we can find a realisation 
$\tilde g$ of $j^rf(\s ;0)$ in ${\mathcal E}_{[r+1]}(n,p),$ 
a sequence $ (x_\nu)_{\nu\geq 1} $ of points of $\mathbb{R}^n \setminus \s$ 
converging to $0 \in \R^n$ such that 
\begin{equation}\label{kk2}
d(x_\nu,\s) \kappa(d \tilde g (x_\nu))+\| \tilde g(x_\nu)\| = 
o(d(x_\nu,\s)^{r+1}).
\end{equation}  

Extracting a subsequence from $(x_\nu)_{\nu\geq 1}$ if necessary, 
one can assume that
$$
\Vert x_{\nu+1}\Vert<\frac{1}{3} d(x_\nu, \s) 
$$ 
which implies, in particular, that $ d(x_\nu,\s) $ decreases,
and that condition \eqref{kk2} implies:
\begin{enumerate}[1)]
\item $| \tilde g_{k}(x_\nu)|=o(d(x_\nu,\s)^{r+1})$, for all  $1\leqq k\leqq p$;
\item $\delta_{\nu}=o(d(x_\nu,\s)^{r})$
where 
$$
\dis\delta_{\nu} := \kappa(d \tilde g(x_\nu)) =
\dist(\mathrm{grad}\tilde g_{j}(x_\nu),\sum_{k \ne j} \mathbb{R} \ 
\mathrm{grad}\tilde g_{k}(x_\nu)) \text{ for some  $j$, $1 \le j \le p.$ } 
$$
\end{enumerate}
Now  adapting the proofs of Theorem \ref{RelativeKuoThm1} and Theorem 
\ref{RelativeKuoThm2}, we first notice that the germ $\eta$,  in this case 
satisfies the inequalities 
\begin{equation}\label{eq-d3}
|\eta_{1}(x)| \leq K(|\tilde g_{1}(x_\nu)|+\epsilon_{\nu}\|x-x_\nu\|^{2})
\precsim o(d(x_\nu,\s)^{r+1})+\epsilon_{\nu}d(x_{\nu},\s)^{2},
\end{equation}
\begin{equation}\label{eq-d4}
|\eta_{k}(x)| \leq K(|\tilde g_{k}(x_\nu)|+\|\lambda_{\nu}^{(k)}\|\|x\ -x_\nu\|) 
\precsim o(d(x_\nu,\s)^{r+1}), \ \ k = 2, \cdots, p.
\end{equation}
Then for a suitable choice of the sequence $(\epsilon_{\nu})_{\eta\geq 1}$, ($\epsilon_{\nu}=o(d(x_\nu,\s)^{r+1})$), we  have $\eta(x)=o(d(x,\s)^{r+1}), $ and then $g=\tilde g-\eta$ is a $C^{r+1}$-realisation of $j^rf(\s ;0).$ We carry on the same argument to contruct in each cases, $n>p$ and $n=p,$ a realisation which contradict Lemma \ref{transv}. 
\end{proof} 

\section{ Rigidity and Relative $SV$-determinacy}

Let $\mathcal{E}(n)^p$, $n \ge p$, be the set of $C^{\infty}$ 
map-germs : $\R^n \to \R^p$ at $0 \in \R^n$, and let $\s$ be 
a germ of closed subset of $\mathbb{R}^n$ such that $0\in \s.$
We say that $f \in \mathcal{E}(n)^p$ is {\em finitely} 
$\s-SV$-{\em determined} (resp. {\em finitely} $\s-V$-{\em determined}) 
if there is a positive integer $k$
such that for any $g \in \mathcal{E}(n)^p$ with $j^k g(\s ;0) = 
j^k f(\s ;0),$ $g$ is $\s-SV$-equivalent (resp. $\s-V$-equivalent) to $f$.
Concerning finite $SV$-determinacy or finite $V$-determinacy in the 
non-relative case, lots of characterisations have been obtained 
(see J. Bochnak - T.-C. Kuo \cite{bochnakkuo}). 

Let $\varphi=(\varphi_{1},\ldots,\varphi_{p}): \R^n \to \R^p$, $n \ge p$,
be a $C^{\infty}$ map-germ at $0 \in \R^n$.
We denote by $I_{K}(\varphi)$ the ideal of ${\mathcal E}(n)$ generated by 
$ \varphi_1,\ldots, \varphi_p$ and the Jacobian determinants
$$
\dis \frac{D(\varphi_1,\ldots, \varphi_p)}{D(x_{i_1},\ldots, x_{i_p})}(x),
\ (1\leq i_1<\ldots<i_p\leq n),
$$
and we let 
$\dis
\dis Z({\varphi},x):=\sum_{1\leq i_1<\ldots<i_{p}\leq n} 
\left|\frac{D(\varphi_1,\ldots, \varphi_p)}
{D(x_{i_1},\ldots,x_{i_{p}})}(x)\right|^2 + \sum_{j=1}^{p} \varphi_j(x)^2.
$

In the case where $n > p$, we define also the ideal of ${\mathcal E}(n)$, 
denoted by $I_{T}(\varphi)$, generated by $ \varphi_1,\ldots, \varphi_p$ and 
the Jacobian determinants
$$
\dis \frac{D(\varphi_1,\ldots, \varphi_p,\rho)}
{D(x_{i_1},\ldots,x_{i_{p+1}})}(x), \ (1\leq i_1<\ldots<i_{p+1}\leq n)
$$
here $\rho(x)=\|x\|^2.$
In the case where $n = p$, we define the ideal $I_{T}(\varphi)$
of ${\mathcal E}(n)$, as the ideal generated by only
$ \varphi_1,\ldots, \varphi_p$.

Let ${\frak m}_{\Sigma}^\infty$ be the ideal of ${\mathcal E}(n)$ consisting of germs $f$ 
such that $j^\infty f(x)=0$ for all $x\in \Sigma$, namely 
$\dis {\frak m}_{\Sigma}^\infty = \{f\in {\mathcal E}(n):\, j^{\infty} f(\s ;0)=0\}.$

Let $r$ be a positive integer, and let
$\varphi=(\varphi_{1},\ldots,\varphi_{p}): \R^n \to \R^p$, $n \geq p$, 
be a $C^r$ map-germ at $0 \in \R^n$.
We denote by $\mathcal{E}_{[r]}(n)$ be the ring of $C^r$ function-germs : 
$\R^n \to \R$ at $0 \in \R^n$, by $\mathcal{E}_{[r]}(n)^p$ the set 
of $C^r$ map-germs : $\R^n \to \R^p$ at $0 \in \R^n$, and by
$\mathcal{E}_{[r]}(n)( \varphi)$ the ideal of $\mathcal{E}_{r}(n)$ 
generated by $\varphi_1,\ldots, \varphi_p$.

\begin{defn} We call $\varphi \in \mathcal{E}(n)^p$ 
$\s$-$C^r$-{\em rigid} if there is a positive integer $k$ for which the 
following holds:\\
for any $\psi\in \mathcal{E}(n)^p$ such that $j^k\varphi=j^k\psi$ 
on $\s,$ there exists 
$\tau \in \mathcal{R}_\s^{\text{fix}}$ such that 
$$ 
\mathcal{E}_{[r]}(n)( \varphi\circ \tau)= \mathcal{E}_{[r]}(n)( \psi). 
$$
\end{defn} 

\begin{defn} Let $I$ be an ideal of $ \mathcal{E}(n).$
We say that $I$ is {\it $\s$-elliptic} if there is $f\in I$ such that 
$\dis |f(x)|\geq Cd(x,\s)^\alpha$
in a neighbourhood of $0$, where $C$ and $\alpha$ are positive constants.
We call such  $f$ an {\em elliptic element} of $I.$
\end{defn} 

\begin{rmk} If  the ideal $I$ is $\s$-elliptic and generated by $f_{1},\ldots, f_{k}$, then $f_{1}^2+\ldots+f_{k}^2$ is  an elliptic element of $I.$
\end{rmk} 

We have the following Lemma, which is a slight modification  of a result of  J.-C. Tougeron and  J. Merrien \cite{MT}.
We give the proof for completeness.

\begin{lem}\label{elliptic}
Let $I$ be  a finitely generated ideal of $ \mathcal{E}(n).$
Then the following conditions are equivalent:
\begin{enumerate}[(1)]
\item $I$ is $\s$-elliptic.
\item ${\frak m}_{\Sigma}^\infty\subset {\frak m}_{\Sigma}^\infty I$
\item ${\frak m}_{\Sigma}^\infty\subset I$
\end{enumerate}

\end{lem} 

\begin{proof}
Let $f_{1},\ldots, f_{l}$ be the generators of $I.$ 

We first show the implication $(1)\Rightarrow (2).$
Let $f=f_{1}^2+\ldots+f_{l}^2.$
By Leibniz formula and the assumption on $I$, in a neighbourhood of $0$, we 
have: for each multi-index $\beta$  there exists $C_{\beta} > 0$ such that 
$$
\left|\displaystyle \frac{\partial^{|\beta|}(1/f)(x)}
{\partial x^{\beta}}\right|\leq\frac{C_{\beta}}{|f(x)|^{(|\beta|+1)}},
$$ 
and there exists $C > 0$ such that 
$$|f(x)|\geq Cd(x,\s)^\alpha.$$

Now, by Proposition VI $  \, 4.2$ of \cite{tougeron},  for any 
$\varphi\in {\frak m}_{\Sigma}^\infty$, 
$\dis \frac{\varphi}{f}\in {\frak m}_{\Sigma}^\infty$.
It follows that $ \varphi\in {\frak m}_{\Sigma}^\infty I.$

The implication $(2)\Rightarrow (3)$ is obvious.

We lastly show the implication $(3)\Rightarrow (1)$.
Suppose that $I$ is not $\s$-elliptic.
Then we can construct 
a sequence of points $x_{k}\in  \mathbb{R}^n\setminus \s,$ converging to 
$0 \in \R^n$ such that 
$\dis  \sum_{i=1}^l |f_{i}(x_{k})|\leq d(x_{k},\s)^{k+1}.$ 
Taking a subsequence if necessary, we may assume that the balls 
$B_{k}=B(x_{k}, \frac{1}{2}d(x_{k},\s))$ are all disjoints.

Let $g_{k}\in \mathcal{E}(n)$ such that $g_{k}(x_{k})=1$
and $g_{k}=0$ on the complement of $B_{k}$ and satisfying:\\ 
for each multi-index $\beta$  there exists a positive constant $C_{\beta}$ 
such that on $B_{k}$
$$
\left|\displaystyle \frac{\partial^{|\beta|}g(x)}
{\partial x^{\beta}}\right|\leq\frac{C_{\beta}}{d(x_{k},\s)^{|\beta|}}.
$$  
Then $\dis \sum_{k\in \mathbb{N}}  g_{k}d(x_{k},\s)^{k}$ converges to a 
function $g\in {\frak m}_{\Sigma}^\infty.$ 
 
By the assumption $(3)$ we have $g\in I$.
Then it follows that there exists $C>0$ such that 
$\dis |g(x_{k})|\leq C \sum_{i=1}^k |f_{i}(x_{k})|$.
Therefore we have $d(x_{k},\s)^k\leq Cd(x_{k},\s)^{k+1}$, which is impossible.
This is a contradiction.
Thus $I$ is $\s$-elliptic
\end{proof} 

As a consequence we have the following proposition:

\begin{prop}\label{prop664}
For $\varphi\in \mathcal{E}(n)^p$ , the following conditions are 
equivalent:

\begin{enumerate}[(1)]
\item There exist $C,\alpha ,\beta > 0$ such that $Z(\varphi,x)
\geq Cd(x,\s)^{\alpha}$ for $\vert x\vert < \beta .$
\item ${\frak m}_{\Sigma}^\infty \subset I_{K}(\varphi).$\\
If moreover  $\s$ is subanalytic and $\varphi$ is analytic, they are also equivalent to:
\item ${\frak m}_{\Sigma}^\infty \subset I_{T}(\varphi).$
\item The set germ at $0$, $\sing(\varphi)\cap \varphi^{-1}(0)$, is contained in $\s.$
\end{enumerate}
\end{prop}
\begin{proof} The equivalence between $(1)$, $(2)$ and $(3)$ follows from 
Theorem \ref{equivKT} and Lemma \ref{elliptic};
$(1)$ implies $(4)$ trivially and the converse is the inequality of 
Lojasiewicz, since $\s$ is subanalytic, $Z(\varphi,x)$ is analytic and 
$\{Z(\varphi,x)=0\}=\sing(\varphi)\cap \varphi^{-1}(0).$
\end{proof} 

\begin{defn}\label{coherent}
A germ of closed subset $\s$  of $ \mathbb{R}^n$ is called {\it coherent} if ${\frak m}_{\Sigma}$ is a finitely generated ideal of $ \mathcal{E}(n).$
\end{defn} 

This definition is inspired by the following result of W. Kucharz proved in \cite{kucharz}: an analytic and semi-algebraic subset  $X$ in an open subset $U$ of $ \mathbb{R}^n$ is coherent if and only if ${\frak m}_{X}$ is a finitely generated ideal of $ \mathcal{E}(n).$ In particular,  $\s=\{0\}$ is coherent.

Let us give a generalisation of the Bochnak-Kuo theorem in \cite{bochnakkuo} 
as follows.

\begin{thm}\label{T1} Let $\s$ be a coherent  germ of closed subset of $\mathbb{R}^n$ such that $0\in \s.$
Then the following conditions are equivalent for 
$\varphi \in \mathcal{E}(n)^p$ where $n> p$:

\begin{enumerate}[(1)]
\item For each $r \in \mathbb{N}$, $\varphi$ is $\Sigma$-$C^{r}$-rigid.
\item $\varphi$ is finitely $\Sigma$-SV-determined.
\item $\varphi$ is finitely $\Sigma$-V-determined.
\item  $I_{K}(\varphi)$ is $\s$-elliptic.
\item ${\frak m}_{\Sigma}^\infty \subset I_{K}(\varphi).$\\
If moreover $\varphi$ is analytic, they are also equivalent to:
\item ${\frak m}_{\Sigma}^\infty \subset I_{T}(\varphi).$

\end{enumerate}
\end{thm}

\begin{proof}
The implications $(1) \implies (2)\implies (3)$ are obvious by definition,
and the equivalence $(4) \Longleftrightarrow (5)$ follows from 
Lemma \ref{elliptic}.
Concerning the equivalence $(5) \Longleftrightarrow (6)$ in the analytic case,
see Proposition \ref{prop664}.

We first show the implication $(5) \Longrightarrow (1)$, namely 
we will show that ${\frak m}_{\Sigma}^\infty \subset I_{K}(\varphi)$ 
implies that for any $r\in \mathbb{N},$
there exists $s\in \mathbb{N}$ such that ${\frak m}_{\Sigma}^s \subset  \mathcal{E}_{[r+1]}(n)( I_{K}(\varphi)).$
Let  $\{f_{1},\ldots, f_{k}\}$ be a system of generators of 
${\frak m}_{\Sigma}.$

Since condition $(5)$ is equivalent to condition $(4)$, for $s$ large enough,  
$\dis\frac{f_{i}^{s}}{{Z}(\varphi,x)}$ is of class $C^{r+1}$ for any 
$i\in\{1,\ldots, k\}.$ 
Hence $f_{i}^{s}\in \mathcal{E}_{[r+1]}(n)( I_{K}(\varphi))$ and then  
${\frak m}_{\Sigma}^{(s-1)k+1} \subset 
\mathcal{E}_{[r+1]}(n)( I_{K}(\varphi)).$
We set $q := (s - 1)k + 1.$

We now show that $j^{2q}(\varphi)$ is $\s$-$ C^{r+1}$-rigid in $\mathcal{E}(n)^p$. Let $\psi\in\mathcal{E}^p$ be any element with $j^{2q}(\psi)=j^{2q}(\varphi)$ on $\s.$ Then $\mathcal{E}(n)(\varphi-\psi)\subset {\frak m}_{\Sigma}^{2q+1}$, hence
\begin{equation}\label{implicit}
\mathcal{E}_{[r+1]}(n)(\varphi-\psi)\subset\mathcal{E}_{[r+1]}(n)( {\frak m}_{\Sigma})(\mathcal{E}_{[r+1]}(n)(I_{K}(\varphi)))^{2}.
\end{equation}
We  first remark that 
$$
\mathcal{E}_{[r+1]}(n)(I_{K}(\varphi)) = \mathcal{E}_{[r+1]}(n)(\varphi)+\mathcal{E}_{[r+1]}(n)(J(\varphi)),
$$ 
where $J(\varphi)$ 
is the Jacobian ideal of $\varphi.$ 
From \eqref{implicit}, we have 
$$\mathcal{E}_{[r+1]}(n)(I_{K}(\varphi))=\mathcal{E}_{[r+1]}(n)(\psi)+ \mathcal{E}_{[r+1]}(n)(J(\varphi))+ \mathcal{E}_{[r+1]}(n)( {\frak m}_{\Sigma})(\mathcal{E}_{[r+1]}(n)(I_{K}(\varphi)))^2$$
and by Nakayama's Lemma we obtain:
$$\mathcal{E}_{[r+1]}(n)(\psi)+ \mathcal{E}_{[r+1]}(n)(J(\varphi))=\mathcal{E}_{[r+1]}(n)(I_{K}(\varphi)).$$
Thus  there exist $$\varphi_{1} \in \mathcal{E}_{[r+1]}(n)( {\frak m}_{\Sigma})(\mathcal{E}_{[r+1]}(n)(J(\varphi)))^{2}\text{ and } \psi_{1}\in \mathcal{E}_{[r+1]}(n)( {\frak m}_{\Sigma})(\mathcal{E}_{[r+1]}(n)(\psi))$$ such that
$\varphi-\psi=\psi_{1}-\varphi_{1}.$

For $x$ and $y$ in $ (\mathbb{R}^n,0) $, we define $F$ by 
$$ 
F(x,y) := \varphi(x+y)-\varphi(x)-\varphi_{1}(x) .
$$
Since, for $ i=1,\ldots,n $,   $F_{i}(x,0)$  belongs to  $\mathcal{E}_{[r+1]}(n)( {\frak m}_{\Sigma})(\mathcal{E}_{[r+1]}(n)(J(\varphi)))^{2}$, by Tougeron's 
Implicit Function Theorem (\cite{tougeron} page 56), there is a map $ g\,:\, (\mathbb{R}^n,0)\longrightarrow(\mathbb{R}^n,0) $ with components in $  \mathcal{E}_{[r+1]}(n)( {\frak m}_{\Sigma})(\mathcal{E}_{[r+1]}(n)(J(\varphi)))$ such that $ F(x,g(x))=0 $. Let $ \tau(x)=x+g(x) $. Clearly $ \tau $ is a germ
of diffeomorphism at the origin, and it coincides with the identity on
$ \s $, namely $\tau\in\mathcal{R}_{\s}^{\mathrm{fix}} .$
Now, by construction, we have  
$$
\varphi(\tau(x))=\varphi(x) + \varphi_{1}(x)=\psi(x) + \psi_{1}(x).
$$
Thus $\mathcal{E}_{[r+1]}(n)(  \varphi\circ\tau)=\mathcal{E}_{[r+1]}(n)(\psi),$ namely $\varphi$ is $\s$-$C^r$-rigid.


It remains to prove the implication $(3) \implies (4).$
Let $\varphi$ be a finitely $\Sigma$-V-determined germ.
We suppose that the germ at $0 \in \R^n$ of $\varphi^{-1}(0)\setminus \s$ 
is not empty.
Let $r\in \mathbb{N}$ such that $\varphi$ is $\s$-V-determined at degree $r$.
Then, using a similar argument to the proof of Lemma \ref{transv}, we can see 
that for all map $g \in \mathcal{E}_{[\infty ]}(n,p)$ with  
$j^rg(\s ;0) = j^r\varphi(\s ;0)$  the germ of 
$g^{-1}(0) \setminus \s$ at $0 \in \R^n$ is not empty and 
a topological manifold of $\R^n$ of codimension $p.$

We suppose that, on the contrary, condition $(4)$ is not satisfied. 
One can then find a sequence $ (x_\nu)_{\nu\geq 1} $ of points of 
$ \mathbb{R}^n $ converging to $0 \in \R^n$ and such that 
\begin{equation}\label{small}
Z(\varphi, x_\nu)=o(d(x_\nu,\s)^s)\text{ for any }s\in \mathbb{N} \text{ and any } \nu\ge 1.
\end{equation}
Extracting a subsequence if necessary, one can assume that
$ \Vert x_{\nu+1}\Vert<\frac{1}{3} d(x_\nu, \s) $, 
which implies, in particular, that $ d(x_\nu,\s) $ decreases. 
As in the proofs of Theorems \ref{RelativeKuoThm1} and 
\ref{RelativeKuoThm2}, we may assume that
$\dis
\kappa (df(x_{\nu}))
= d\ (\mathrm{grad}\varphi_{1}(x_\nu),\sum_{k=2}^p \mathbb{R}\,
\mathrm{grad}\varphi_{k}(x_\nu)).
$
Let
$\dis
\delta_{\nu} := d\ (\mathrm{grad}\varphi_{1}(x_\nu),\sum_{k=2}^p \mathbb{R}\,
\mathrm{grad}\varphi_{k}(x_\nu)).
$
Since
$$Z(\varphi, x_{\nu})=\displaystyle \sum_{k=1}^{p}\varphi_{k}^{2}(x_{\nu})+\sum_{1\leq i_1<\ldots<i_{p}\leq n} 
\left|\frac{D(\varphi_1,\ldots, \varphi_p)}
{D(x_{i_1},\ldots,x_{i_{p}})}(x_{\nu})\right|^2 \geq  \sum_{k=1}^{p}\varphi_{k}^{2}(x_{\nu})+\delta_{\nu}^{2p},$$
condition \eqref{small} implies:
\begin{enumerate}[1)]
\item $|\varphi_{k}(x_\nu)|=o(d(x_\nu,\s)^{s})$, for all $s\in \mathbb{N}$ and $1\leqq k\leqq p$;
\item $\delta_{\nu}=o(d(x_\nu,\s)^{s})$, for all $ s\in \mathbb{N}.$

\end{enumerate}

Now we apply the Bochnak-Kuo Lemma in \cite{bochnakkuo} 
(see Lemma \ref{bochnakkuo1} with Remark \ref{bochnakkuo2})
with $u_{\nu}^{(k)}=\mathrm{grad}\varphi_{k}(x_\nu)$ and $\alpha_{\nu}=d(x_\nu,\s)$, to find for each $\nu\in \mathbb{N},$   $p-1$ vectors,  
$\lambda_{\nu}^{( 2)}, \ldots, \lambda_{\nu}^{( p)}\in \mathbb{R}^{n}$ such that:

\begin{enumerate}[$\bf (a)$]
\item $\dis
\|\lambda_{\nu}^{(k)}\|=o(d(x_\nu,\s)^{s}),\ k=2,,\ldots,\ p;
$
\item  $\mathrm{grad}\varphi_{2}(x_{\nu})+\lambda_{\nu}^{( 2)},\ \ldots, \mathrm{grad}\varphi_{p}(x_{\nu})+\lambda_{\nu}^{( p)}$ are linearly independant in  $\mathbb{R}^{n}$;
\item $\dis \mathrm{grad}\varphi_{1}(x_{\nu})\in\sum_{k=2}^{p}\mathbb{R}\ (\mathrm{grad}\varphi_{2}(x_{\nu})+\lambda_{\nu}^{( 2)})$.
\end{enumerate}

Let $\psi: \mathbb{R}\rightarrow[0,1]$ be a $C^{\infty}$ function such that 
$\psi(t) = 1$ in a neighbourhood of $0\in \mathbb{R}$ and $\psi(t)=0$ for 
$|t|\displaystyle \geqq\frac{1}{4}.$ 
Let $f:\mathbb{R}^{n}\rightarrow \mathbb{R}$  a smooth flat function such that 
$f^{-1}(0)=\s.$  
We define a germ $\eta=(\eta_{1},\ldots,\eta_{p})$ by:
$$
\eta_{1}(x)=\sum_{\nu=1}^{\infty}\psi\circ f\left(\frac{x-x_\nu}{\|x_\nu\|}\right)(\varphi_{1}(x_\nu)+\epsilon_{\nu}\|x-x_\nu\|^{2}),
$$
$$
\eta_{k}(x)=\sum_{\nu=1}^{\infty}\psi\circ f\left(\frac{x-x_\nu}{\|x_\nu\|}\right)(\varphi_{k}(x_\nu)-\langle\lambda_{\nu}^{(k)},(x\ -x_\nu)\rangle),\ k=2,\ \ldots,\ p.
$$
Then 
\begin{enumerate}[i)]
\item If we choose $\epsilon_{\nu} >0$ such that for each $s\in \mathbb{N},\ \epsilon_{\nu}=o(d(x_\nu,\s)^{s}),$ then $\eta= (\eta_{1},\ \ldots,\ \eta_{p})$ is of class $C^{\infty}$,
\item $\eta$ is  flat on $\s$;
\item  For each $\nu\in \mathbb{N},\ (\varphi-\eta)(x_\nu)=0$.

\end{enumerate}

Let $g=\varphi-\eta$. 
We shall show that   we can  choose $\epsilon_{\nu},$  
$\epsilon_{\nu}=o(d(x_\nu,\s)^{s})$, such that near each $x_\nu,$ 
$ g^{-1}(g(x_{\nu}))=g^{-1}(0)$ is not a topological manifold of codimension 
$p$, if $n>p$.\\
By condition $\bf (b)$, there is  a small neighbourhood  $V_{\nu}$ of $x_{\nu}$, such that the set 
$$M_{\nu}=\{x\in V_{\nu}:\, \varphi_{k}(x)-\eta_{k}(x)=0,\ k=2,\ \ldots,p\}$$ is a smooth manifold of codimension $p-1.$

From condition $\bf (c),$ for each $\nu\in \mathbb{N},$  there are real numbers  $a_{2,\nu},\ldots, a_{p,\nu}$ such that,
$$
\mathrm{grad}\varphi_{1}(x_{\nu})=\sum_{k=2}^{p}a_{k,\nu}(\mathrm{grad}\varphi_{k}(x_{\nu})+\lambda_{\nu}^{( k)}).
$$
Choose now $\epsilon_{\nu}=o(d(x_\nu,\s)^{s})$ such that $x_\nu$ is a non-degenerate critical point of
$$
h_{\nu}(x)=\varphi_{1}(x)-\eta_{1}(x)+\sum_{k=2}^{p}a_{k,\nu}(\eta_{k}(x)-\varphi_{k}(x)).
$$
Then
$\dis 
g^{-1}(0)\cap V_{\nu}=\{x\in V_{\nu}:\, g_{1}(x)=g_{2}(x)=\ldots=g_{p}(x)=0\}
=\{x\in M_{\nu}:h_{\nu}(x)=0\}.
$
By the choice of  $\epsilon_{\nu}$, this set is the intersection of  the locus of a non-degenerate quadratic form $h_{\nu}^{-1}(0)$  with a  codimension $p-1$ manifold $M_{\nu}.$ Then if it is a topological manifold, necessarily it is reduces to a point. Now if $n-p\ge1,$  $g^{-1}(0)$ cannot be a topological manifold of codimension $p.$
This is a contradiction.
Thus the implication $(3) \implies (4)$ is shown.
\end{proof} 
\begin{thm}\label{T2} Let $\s$ be a  coherent subanalytic  germ of closed subset at $0\in\mathbb{R}^n.$  Then the following conditions are equivalent for 
analytic germ $\varphi : (\mathbb{R}^n,0)\to (\mathbb{R}^n,0):$
\begin{enumerate}[(1)]
\item For each $r \in \mathbb{N}$, $\varphi$ is $\Sigma$-$C^{r}$-rigid.
\item $\varphi$ is finitely $\Sigma$-SV-determined.
\item $\varphi$ is finitely $\Sigma$-V-determined.
\item  $I_{K}(\varphi)$ is $\s$-elliptic.
\item ${\frak m}_{\Sigma}^\infty \subset I_{K}(\varphi).$

\item ${\frak m}_{\Sigma}^\infty \subset I_{T}(\varphi).$

\end{enumerate}
\end{thm}
\begin{proof} The proofs are the same as in the previous theorem except for 
the implication $(3) \implies (4)$, where we conclude by the following:
since $n=p$ and the set-germ $f^{-1}(0)\setminus \s$ is not empty, this 
contradicts the $\s$-V-sufficiency (see  Corollary \ref{manifold}).
\end{proof} 


\section{Relative $\mathcal{K}$ equivalence}
Let $\mathcal{E}({n})$ be the local ring of germs of $C^{\infty}$ functions $f:(\mathbb{R}^{n},\ 0)\rightarrow \mathbb{R}$ with maximal ideal ${m}_{n}$. 
For a germ of closed subset $\s$ of $\mathbb{R}^n$ such that $0\in \s,$ we suppose moreover that $\s$ is coherent (see Definition \ref{coherent}). 
We now generalise Mather's notion of contact equivalence (see \cite{Mather}); 
we say that two map germs $f$ and $g$ 
$:(\mathbb{R}^{n},\ 0)\rightarrow(\mathbb{R}^{p},\ 0)$ are  
$\mathcal{K}_{\s}$ {\em equivalent} if there exists a germ of diffeomorphism $h\in  \mathcal{R}_{\s}^{fix}$ and a $C^{\infty}$ germ $ A:(\mathbb{R}^{n},\ 0)\rightarrow GL(\mathbb{R},\ p)$ such that $f=A\cdot g \circ h;$
 here $\cdot$ denotes matrix multiplication of $A$ by the vector-valued function $g\circ h$ in $\mathbb{R}^{p}$. 
Now let $f= (f_{1},\ \ldots,f_{p})$ be an element in 
$m_{n}\mathcal{E}({n})^{p},$ let $J_{\mathcal{C}}(f)$ be the ideal in $\mathcal{E}({n})$ generated by $f_{1},\ \ldots,\ f_{p}$ and let 
$\displaystyle \langle\frac{\partial f}{\partial {x}_{1}},\ \ldots,\ \displaystyle \frac{\partial f}{\partial x_{n}}\rangle$ be the $\mathcal{E}({n})$ submodule of $\mathcal{E}({n})^{p}$ generated by the partial derivatives $\displaystyle \frac{\partial f}{\partial x_{i}},\ i=1,\ \ldots,\ n$. In analogy with the $ \mathcal{K}$ tangent space of a map germ  (see  \cite{Mather} ), we define the $\mathcal{E}({n})$ submodule of $\mathcal{E}({n})^{p}$ :
$$
T\mathcal{K}_{\s}(f):= \mathfrak{m}_{\s}\langle\frac{\partial f}{\partial x_{1}},\ \ldots,\ \frac{\partial f}{\partial x_{n}}\rangle + 
J_{\mathcal{C}}(f)\mathcal{E}({n})^{p}.
$$
\begin{prop}\label{P1}
A necessary and sufficient condition for $f$ to be $\mathcal{K}_{\s}$ determined by a finite jet (resp. $\infty$-jet) is that for some $ k<\infty$, $\mathfrak{m}_{\s}^{k}\mathcal{E}({n})^{p}\subset T\mathcal{K}_{\s}f$
(resp. $\mathfrak{m}_{\s}^{\infty}\mathcal{E}({n})^{p}\subset T\mathcal{K}_{\s}f$ ).
\end{prop} 

\begin{proof}
Assume that $f\in \mathcal{E}({n})^{p}$ satisfies the condition 
$\mathfrak{m}_{\s}^{\infty}\mathcal{E}({n})^{p}\subset T\mathcal{K}_{\s}f$, 
denoted by (t) as in \cite{wall} . Let $h\in \mathfrak{m}_{\s}^{\infty}\mathcal{E}({n})^{p}$. It is clear that we only have to prove that $f$ and $f+h$ are $\mathcal{K}_{\s}$ equivalent. Define $F: (\mathbb{R}\times \mathbb{R}^n,\ 0)\rightarrow(\mathbb{R}^{p},\ 0)$ by $F(t,\ x) =f(x)+th(x)$.

\begin{lem}
{\it If} $f$ {\it satisfies condition} (t) {\it then we can find a germ of} a 
{\it smooth vector field} $X$ {\it around} $(0,0)$ {\it in} 
$\mathbb{R}\times \mathbb{R}^{n}$ {\it of the following form}:

$X =\displaystyle \frac{\partial}{\partial t}+\sum_{j=1}^{n}X_{j}(t,\ x) \displaystyle \frac{\partial}{\partial x_{j}}$ 
where $X_{j}(t,\ x)=0$ for $x\in \Sigma$ and such that: \\ $DF(X)\in J_{ \mathcal{C}} (F) \mathcal{E}({n+1})^{p}$.
\end{lem} 
\begin{proof} From the coherency condition, the following $ \mathcal{E}(n)$ 
module
$$
\mathfrak{m}_{\s}\langle\frac{\partial f}{\partial x_{1}},\ldots,\ \frac{\partial f}{\partial x_{1}}\rangle+J_{ \mathcal{C}}(f) \mathcal{E}(n)^{p}\ =T \mathcal{K}_{\s}f
$$
is finitely generated.
Considering $\mathcal{E}({n})$ as a subset of $\mathcal{E}({n+1})$ it follows 
from condition (t) that:
$$
m_{n+1}\mathfrak{m}_{\s}^{\infty} \mathcal{E}({n+1})^{p}\subset m_{n+1}\left(\mathfrak{m}_{\s} \mathcal{E}({n+1})\langle\frac{\partial f}{\partial x_{1}}, \ldots,\ \frac{\partial f}{\partial x_{n}}\rangle+J_{ \mathcal{C}}(f) \mathcal{E}({n+1})^{p}\right).
$$
Since $F-f\in m_{n+1}\mathfrak{m}_{\s}^{\infty}\mathcal{E}({n})^{p},$ 
this implies that

$$
\mathfrak{m}_{\s}\mathcal{E}({n+1})\displaystyle \langle\frac{\partial F}{\partial x_{1}} , \ldots, \displaystyle \frac{\partial F}{\partial {x}_{n}}\rangle+J_{ \mathcal{C}}(F) \mathcal{E}({n+1})^{p}\subseteq \mathfrak{m}_{\s} 
\mathcal{E}({n+1})\langle\frac{\partial f}{\partial x_{1}} , \ldots, \displaystyle \frac{\partial f}{\partial x_{n}}\rangle+J_{ \mathcal{C}}(f) \mathcal{E}({n+1})^{p}.
$$

On the other hand we have
$$
\mathfrak{m}_{\s} \mathcal{E}({n+1})\displaystyle \langle \frac{\partial f}{\partial x_{1}},\ \ldots,\ \displaystyle \frac{\partial f}{\partial {x}_{n}}\rangle+J_{ \mathcal{C}}(f) \mathcal{E}({n+1})^{p}
$$
$$
\subseteq \mathfrak{m}_{\s}\mathcal{E}({n+1})\langle\frac{\partial F}{\partial x_{1}},\ldots, \frac{\partial F}{\partial x_{n}}\rangle+J_{ \mathcal{C}}(F)\mathcal{E}({n+1})^{p}+m_{n+1}\mathfrak{m}_{\s}^{\infty}\mathcal{E}({n+1})^{p}
$$
$$
\subseteq \mathfrak{m}_{\s}\mathcal{E}({n+1})\langle\frac{\partial F}{\partial x_{1}}\cdots,\ \frac{\partial F}{\partial x_{n}}\rangle+J_{ \mathcal{C}}(F)\mathcal{E}({n+1})^{p}
$$
$$
+m_{n+1}\left(\mathfrak{m}_{\s}\mathcal{E}({n+1})\langle\frac{\partial f}{\partial x_{1}},\ \ldots,\ \frac{\partial f}{\partial x_{n}}\rangle+J_{ \mathcal{C}}(f)\mathcal{E}({n+1})^{p}\right).
$$
Hence Nakayama's Lemma gives that
$$
\mathfrak{m}_{\s}\mathcal{E}({n+1})\displaystyle \langle\frac{\partial f}{\partial x_{1}} , \ldots, \displaystyle \frac{\partial f}{\partial x_{n}}\rangle+J_{ \mathcal{C}}(f)\mathcal{E}({n+1})^{p}=\mathfrak{m}_{\s}\mathcal{E}({n+1})\langle\frac{\partial F}{\partial x_{1}} , \ldots, \displaystyle \frac{\partial F}{\partial x_{n}}\rangle+J_{ \mathcal{C}}(F)\mathcal{E}({n+1})^{p}.
$$

From condition (t) it follows that
$$
h\in \mathfrak{m}_{\s}^{\infty}\mathcal{E}({n+1})^{p}\subseteq \mathfrak{m}_{\s}\mathcal{E}({n+1})\langle\frac{\partial F}{\partial {x}_{1}},\ldots, \frac{\partial F}{\partial x_{n}}\rangle+J_{ \mathcal{C}}(F)\mathcal{E}({n+1})^{p}.
$$
This shows that we can find germs $X_{j}\in\mathfrak{m}_{\s}\mathcal{E}({n+1})$ such that
$$
h+\sum_{j=1}^{n}X_{j}\frac{\partial F}{\partial x_{j}}\in J_{ \mathcal{C}}(F)\mathcal{E}({n+1})^{p}.
$$
Then $ \displaystyle X(t, \ x)=\frac{\partial}{\partial t}+\sum_{j=1}^{n}X_{j}(t,\ x)  \frac{\partial}{\partial x_{j}}$ satisfies the conditions of this lemma.
\end{proof} 

Now, integrate the vector field $X$ in the lemma above around $(0,0)$ 
in $\mathbb{R}\times \mathbb{R}^{n}$. We get a family $\{h_{t}\}$ of diffeomorphisms $(\mathbb{R}^{n},\ 0)\rightarrow(\mathbb{R}^{n},\ 0)$ such that $h_{t}|_{\s}=Id$. The condition $DF(X)\in J_{ \mathcal{C}}(F)\mathcal{E}({n+1})^{p}$ gives that we can find a $p\times p$ matrix $A(t,\ x)$ with entries in $\mathcal{E}({n+1})$ such that
$$
\frac{d}{dt} F(t,\ h_{t}(x)) = DF(t,\ h_{t}(x)) X(t,\ h_{t}(x)) 
= A(t,\ h_{t}(x)) \cdot F(t,\ h_{t}(x))
$$
which gives that $F(t,\ h_{t}(x))$ is a solution of the differential equation 
${y'}=A(t,\ h_{t}(x)) \cdot y$ with initial condition $y(0) = f(x)$ for fixed 
$x$.
Since the solution of this differential equation is unique and smooth in $x$ 
and $t$ of form $y(x,t)=A(t, x) \cdot y(x, 0)$ where $A(t,\ x)$ is an invertible matrix, we can conclude that $F(t,\ h_{\mathrm{t}}(x)) = A(t, x) \cdot f(x)$. 
Since this holds in a neighbourhood of $(0,0)$ in 
$\mathbb{R}\times \mathbb{R}^{n}$ we can conclude that $f$ and $f+th$ are 
$\mathcal{K}_{\s}$ equivalent for small $\mathrm{t}$. 
Now fix an arbitrary $t_{0}\in[0,1]$, and let 
$f_{t_{0}}\in \mathcal{E}({n})^{p}$ denote $f+t_{0}h$. 
From condition (t) it follows easily that 
$T \mathcal{K}_{\s}f_{t_{0}}\subset T\mathcal{K}_{\s}f$ and that 
$T\mathcal{K}f\subseteq T\mathcal{K}_{\s}f_{t_{0}}+m_{n}(T\mathcal{K}_{\s}f)$. 
Therefore Nakayama's Lemma gives that 
$T\mathcal{K}_{\s}f=T\mathcal{K}_{\s}f_{0}$. 
Thus $T\mathcal{K}_{\s}f_{t_{0}}$ also satisfies condition (t), and it follows 
from above that $f_{t_{0}}$ and $f+th$ are $ \mathcal{K}_{\s}$ equivalent when 
$|t-t_{0}|$ is small. 
Connectedness of $[0,1]$ gives that $f$ and $f+h$ are $ \mathcal{K}_{\s}$ equivalent.
\end{proof}
Let us also introduce the notion of $C^{r}$-$\mathcal{K}_{\s}$ equivalence $ 0\leq r\leq\infty$ which is the analogue of ordinary $\mathcal{K}_{\s}$ equivalence just using $C^{k}$ diffeomorphisms instead of $C^{\infty}$ diffeomorphisms in the definition of $\mathcal{K}_{\s}$ equivalence.
Let $J_{\mathcal{R}}(f)$ be the ideal in $\mathcal{E}({n})$ generated by the $p\times p$ minors of $Df$ and put $J_{\mathcal{K}}(f)=J_{\mathcal{R}}(f)+J_{\mathcal{C}}(f)$. 
Following \cite{brodersen} we  denote:
\begin{enumerate}[]
\item $(a_{r})$ for  $0\leq r\leq\infty,$  $f$ is infinitely 
$C^{r}$-$\mathcal{K}_{\s}$ determined.
\item $(\mathrm{b}_{r})$ for $0\leq r \le \infty,$ $f$ is finitely 
$C^{r}$-$\mathcal{K}_{\s}$ determined.
\item (c) $J_{\mathcal{K}}(f)$ is $\s$-elliptic.
\item (t) $ \mathfrak{m}_{\s}^{\infty}\mathcal{E}({n})^{p}\subset T\mathcal{K}_{\s}f.$
\item $({v}_{1})$ $f$ is infinitely $\s$-$ V$ determined.
\item $({v}'_{1})$ $f$  is infinitely $\s$-SV-determined.
\item $({v}_{2})$ $f$ is finitely $\s$-$V$-determined.
\item $({v}'_{2})$ $f$ is finitely $\s$-$SV$-determined.
\end{enumerate}
Now we have:
\begin{thm}\label{thmequiv} 
For $f\in \mathcal{E}({n})^{p},$ $n>p,$ the conditions
$$
({a}_{{r}}) \ (0\leq r\leq\infty), \ ({b}_{r}) \ (0\leq r<\infty), \ (c), 
\ (t), \ ({v}_{1}), \ ({v}'_{1}), \ ({v}_{2}) \ \text{and} \ ({v}'_{2})
$$ 
are all equivalent.
\end{thm}
In Proposition \ref{P1}, we have proved the equivalence 
$(t)\iff({a}_{{\infty}}).$ 
By definition, the implications $({a}_{\infty})\implies({a}_{{r}})\implies
({v}_{1}),$ $({v}_{{2}})\implies({v}_{1})$ and $({v}'_{{2}})\implies
({v}'_{1})$ are obvious. 
The  proof of $(3) \implies (4)$ of Theorem  \ref{T1} can be easily adapted 
to prove the implication $(v_{1}) \implies (c),$ and 
Lemma \ref{elliptic} gives the implication $(c)\implies (t).$
In addition,  the notion of $C^{r}$ rigid is equivalent to $(\mathrm{b}_{r})$. 
Therefore the equivalences 
$({c})\iff({v}_{2})\iff({v}'_{2}) \iff({b}_{r}),$ ($r < \infty$),
are proved in the same way as in Theorem \ref{T1}. 
Thus Theorem \ref{thmequiv} is established.
\begin{rmk}
We may notice that in Theorem \ref{thmequiv}, we consider only $(b_{r})$ with $0\leq r<\infty$;  in fact in general $(a_{\infty})$ doesn't implies $(b_{\infty})$, even in the absolute case. For example, the  $C^\infty$ function germ at the origine,
$f(x,y)=(x^2+y^2)^2,$ is  infinitely $C^{\infty}$-$\mathcal{K}$-determined, because it's jacobian ideal is elliptic (see  \cite{W1}  Theorem 1.2 for this characterisation of infinite determinacy)
but it is not finitely $C^{\infty}$-$\mathcal{K}$-determined, since the zeros set of each representative of it's complexification has  singular points arbitrary close to the origine
(see \cite{wall} Proposition 1.7 for this characterisation of finite determinacy). 
\end{rmk}


\end{document}